\documentclass[preprint,3p,10pt]{elsarticle}

\usepackage{amsmath}
\usepackage{amssymb}
\usepackage{amsthm}
\usepackage{verbatim}
\usepackage{marvosym}
\usepackage{color}
\usepackage[table]{xcolor}
\usepackage{soul}

\newtheorem{theorem}{Theorem}[section]
\newtheorem{lemma}[theorem]{Lemma}
\newtheorem{corollary}[theorem]{Corollary}
\newtheorem{proposition}[theorem]{Proposition}

\theoremstyle{definition}
\newtheorem*{acknowledgement}{Acnowledgement}

\numberwithin{equation}{section}

\usepackage{amssymb}

\journal{Journal of Functional Analysis}

\begin{document}

\begin{frontmatter}




\title{Extended eigenvalues for  Ces\`aro operators}
\author{Miguel Lacruz\corref{miguel}}
\address{Departamento de An\'alisis Matem\'atico, Facultad de Matem\'aticas, Universidad de Sevilla, Avenida Reina Mercedes s/n, 41012 Seville (Spain)}
\ead{lacruz@us.es}
\author{Fernando Le\'on-Saavedra\corref{}}
\address{Departamento de Matem\'aticas, Universidad de C\'adiz, Avda. de la Universidad s/n, 11405 \indent Jerez de la  Frontera, C\'adiz (Spain)}
\ead{fernando.leon@uca.es}
\author{Srdjan Petrovic\corref{}}
\address{Department of Mathematics, Western Michigan University, Kalamazoo, MI 49008 (USA)}
\ead{srdjan.petrovic@wmich.edu}
\author{Omid Zabeti\corref{}}
\address{Department of Mathematics, Faculty of Mathematics, University of Sistan and Baluchestan,  \indent P.O. Box 98135-674, Zahedan (Iran)}
\ead{o.zabeti@gmail.com}
 \cortext[miguel]{Corresponding author}

\begin{abstract}
A complex scalar \(\lambda\) is said to be an {\em extended eigenvalue} of   a bounded linear operator   \(T\)  on a complex Banach space if there is a nonzero operator \(X\) such that \(TX= \lambda XT.\) Such an operator \(X\) is called an {\em extended eigenoperator} of \(T\) corresponding to  the extended eigenvalue~\(\lambda.\) 

The purpose of this paper is  to give a description of the extended eigenvalues  for the  discrete Ces\`aro operator \(C_0,\) the finite continuous   Ces\`aro operator \(C_1\)  and the infinite continuous  Ces\`aro operator \(C_\infty\)  defined on the complex Banach spaces \(\ell^p,\) \(L^p[0,1]\) and \(L^p[0,\infty)\) for \(1 < p <\infty\) by the expressions  
\begin{align*}
(C_0f)(n) \colon & = \frac{1}{n+1} \sum_{k=0}^n f(k),\\
(C_1f)(x) \colon & = \frac{1}{x} \int_0^x f(t)\,dt,\\
(C_\infty f)(x) \colon & = \frac{1}{x} \int_0^x f(t)\,dt.
\end{align*}
It is shown that the set of extended eigenvalues for \(C_0\) is the interval \([1,\infty),\) for \(C_1\) it is the interval \((0,1],\)  and  for \(C_\infty\) it reduces to the singleton~\(\{1\}.\) 
\end{abstract}


\end{frontmatter}




\newpage

\def\l@section{\@dottedtocline{1}{1em}{2em}}

\tableofcontents

\section{Introduction}
We shall represent by \({\mathcal B}(E)\)  the algebra of all bounded linear operators  on a complex Banach space \(E.\)  A complex scalar \(\lambda\) is  said to be an {\em extended eigenvalue} of an operator \(T \in {\mathcal B}(E)\) provided that there is a nonzero operator \(X \in {\mathcal B}(E)\) such that \(TX= \lambda XT,\) and in that case \(X\) called an {\em extended eigenoperator} of \(T\) corresponding to  the extended eigenvalue \(\lambda.\) We shall represent by \(\{T\}^\prime\) the commutant of an operator~\(T,\) i.e., the set of operators that commute with~\(T,\) or in other words, the family of all the extended eigenoperators for \(T\) corresponding to the extended eigenvalue \(\lambda=1.\)

Recently, the study of the extended eigenvalues   for some classes of operators has received a considerable amount of attention \cite{BLP,BP,BS,Lam,Lau,P,P2, S}. 

The purpose of this paper is  to describe the set of the extended eigenvalues  for the  discrete Ces\`aro operator \(C_0,\) the finite continuous   Ces\`aro operator \(C_1,\)  and the infinite continuous  Ces\`aro operator \(C_\infty\)  defined on the complex Banach spaces \(\ell^p,\) \(L^p[0,1]\) and \(L^p[0,\infty)\) for \(1 < p <\infty\) by the expressions

\begin{eqnarray}
\label{cesdisc}
(C_0f)(n) \colon = \frac{1}{n+1} \sum_{k=0}^n f(k),
\end{eqnarray}
\begin{eqnarray}
\label{cesfin}
(C_1f)(x) \colon = \frac{1}{x} \int_0^x f(t)\,dt,
\end{eqnarray}
\begin{eqnarray}
\label{cesinf}
(C_\infty f)(x) \colon = \frac{1}{x} \int_0^x f(t)\,dt.
\end{eqnarray}

It is shown that the set of extended eigenvalues for \(C_0\) is the interval \([1,\infty),\) for \(C_1\) is the interval \((0,1],\)  and  for \(C_\infty\) is the singleton \(\{1\}.\) The notion of an operator with rich point spectrum is introduced and it is  shown   that  the geometry of the point spectrum  for such an operator determines its extended eigenvalues. Then, it is shown that both \(C_1\) and \(C_0^\ast\) have rich point spectrum. Further, it is shown  that a bilateral weighted shift whose point spectrum has non empty interior and the adjoint of an analytic Toeplitz operator with non constant symbol are further examples of operators with rich point spectrum. Then,  this result is applied to obtain information on the extended eigenvalues of those operators. Finally,  a factorization is provided for the extended eigenoperators   of a Hilbert space operator under certain conditions.

The paper is organized as follows.

In section \ref{finite} we show that  every \(\lambda \in (0,1]\)  is an  extended eigenvalue for \(C_1\) on  \(L^2[0,1]\) and    the Euler operator is a corresponding extended eigenoperator. Moreover, any extended eigenoperator for \(C_1\) on   \(L^2[0,1]\)  factors as the product of the Euler operator, a Toeplitz matrix, and a power of a backward unilateral  shift of multiplicity one.

In section~\ref{sect:rich} we introduce the notion of an operator with  rich point spectrum. We show that if \(\lambda\) is an extended eigenvalue of  an operator \(T\) with rich point spectrum  then \(\lambda\)  multiplies \({\rm int}\, \sigma_p(T),\) the interior of the point spectrum of \(T,\) into \({\rm clos}\, \sigma_p(T),\) the closure of the point spectrum of \(T.\)  We  show that  both \(C_1\) and \(C_0^\ast\) have rich point spectrum and we apply this geometric  result to prove that for every \(1<p<\infty\) we have
\begin{enumerate}
\item if \(\lambda\) is an extended eigenvalue for \(C_1\) on \(L^p[0,1]\)  then  \(0 < \lambda \leq 1,\) 
\item  if \(\lambda\) is an extended eigenvalue for \(C_0\) on \(\ell^p\)  then  \(\lambda \geq 1.\)
\end{enumerate}

In section \ref{cesp} we show that  every \(\lambda \in (0,1]\) is an extended eigenvalue for \(C_1\) on \(L^p[0,1]\) and that  a certain weighted composition operator is a corresponding extended eigenoperator.

In section \ref{sect:disc} we show  when \(p=2\) that if \(\lambda\) is  real  with \(\lambda \geq 1\) then \(\lambda\) is an extended eigenvalue for \(C_0.\) 
 
In section \ref{sec:bilateral} we show that if the point spectrum of a bilateral weighted  shift \(W\) has non empty interior  then \(W\) has rich point spectrum, and as a consequence, the set of the extended eigenvalues for \(W\) is the unit circle.

In section \ref{toeplitz} we show that a result of Deddens \cite{D} about  extended eigenvalues of an analytic Toeplitz operators can be regarded as a special case of our main result in section \ref{sect:rich}.

In section \ref{sec:factor} we show under certain conditions that  if  \(\lambda\) is an extended eigenvalue for an operator \(T\) on  a Hilbert space then there is  a particular extended eigenoperator \(X_0\) corresponding to    \(\lambda\) such that every extended eigenoperator \(X\) corresponding to \(\lambda\) factors as  \(X=X_0R\) for some  \(R \in \{T\}^\prime.\)

In section \ref{infinite} we show that the family of  the extended eigenvalues for  \(C_\infty\) on the complex Hilbert space \(L^2[0,\infty)\) reduces to the singleton \(\{1\}.\) 

In section \ref{sec:cinftyp} we show that the family of the extended eigenvalues for  \(C_\infty\) on the complex Banach space \(L^p[0,\infty),\) for \(1<p<\infty,\)  reduces to the singleton~\(\{1\}.\)

 
\section{The finite continuous Ces\`aro operator on Hilbert space}
\label{finite}

Brown, Halmos and Shields \cite{BHS} proved in the Hilbertian case that \(C_1\) is indeed a bounded linear operator, and they also proved that \(I-C_1^\ast\) is unitarily equivalent to a unilateral  shift of multiplicity one.  

Recall that a bounded linear operator \(S\) on a complex Hilbert space \(H\)  is a {\em unilateral  shift of multiplicity one} provided that  there is an orthonormal basis \((e_n)\) of \(H\)  such that   \(Se_n=e_{n+1}\) for all \(n \in \mathbb N.\)  It is easy to see that the adjoint of a such a unilateral shift satisfies \(S^\ast e_0=0\) and \(S^\ast e_n = e_{n-1}\) for all \(n \geq 1.\)

Consider a unilateral shift of multiplicity one \(S \in {\mathcal B}(L^2[0,1])\) and  a unitary operator  \(U \in {\mathcal B}(L^2[0,1])\)  such that \(I-C_1^\ast=U^\ast SU.\) We have \(C_1=U^\ast (I-S^\ast)U,\) and since the extended eigenvalues are preserved under similarity in general, and under unitary equivalence in particular, it follows that the extended eigenvalues of \(C_1\) are precisely the extended eigenvalues of \(I-S^\ast,\) and the extended eigenoperators of \(C_1\) are in one to one correspondence with   the extended eigenoperators of \(I-S^\ast\)  under conjugation with  \(U.\)

We shall use repeatedly the following elementary, standard fact.

\begin{lemma}
\label{vito}
The point spectrum of \(S^\ast\) is the open  unit disc \(\mathbb D.\) More precisely, every \(\lambda \in \mathbb D\)  is a simple eigenvalue of \(S^\ast,\) and a corresponding eigenvector \(f\)  is given by the expression 
\begin{eqnarray}
f = \sum_{n=0}^\infty  \lambda^n e_n.
\end{eqnarray}
\end{lemma}
\noindent
Now we are ready to describe the set of the extended eigenvalues for \(I-S^\ast.\) Our  first goal is  to show that the   interval \((0,1]\) is contained in the set of   the extended eigenvalues for \(I-S^\ast,\) and to exhibit a  corresponding extended eigenoperator. We shall prove that a particular  extended eigenoperator is the  Euler operator. 

It is convenient now  to  have a digression   about the Euler operator and the discrete Ces\`aro operator.  We   follow the discussion in the paper of Rhoades ~\cite{R}. Recall that the discrete Ces\`aro operator \(C_0\) is defined on  \(\ell^2\) by the sequence of arithmetic means (\ref{cesdisc}).

Let \(\lambda \in \mathbb C.\)  The {\em Euler operator} \(E_\lambda\) is defined on \(\ell^2\) by the  binomial means
\begin{eqnarray}
\label{euler}
(E_\lambda f)(n)    =  \sum_{k=0}^n \binom{n}{k} \lambda^k (1-\lambda)^{n-k} f(k), \quad n \in \mathbb N.
\end{eqnarray}
Let \((\mu_k)\) be a sequence of complex scalars and let \(\Delta\) denote the  {\em forward difference operator} defined by 
\begin{eqnarray}
\Delta \mu_k=\mu_k - \mu_{k+1}.
\end{eqnarray}
A {\em Hausdorff matrix} is an infinite  matrix \(A=(a_{nk})\) 
whose entries are given by the expression
\begin{eqnarray}
a_{nk} = 	\left \{
		\begin{array}{rl}
		\displaystyle{\binom{n}{k} \Delta^{n-k} \mu_k} & \text{if } 0 \leq k \leq n,\\
		0 & \text{if }k>n.
		\end{array}
		\right .
\end{eqnarray}
The sequence \((\mu_k)\) is called the {\em generating sequence} for the Hausdorff matrix \(A\) and it is determined by the diagonal entries of \(A.\) 
The  {\em Hausdorff operator} associated with a Hausdorff matrix \(A=(a_{nk})\) is  defined by the expression
\begin{eqnarray}
(Af)(n)= \sum_{k=0}^na_{nk}f(k).
\end{eqnarray}
The discrete Ces\`aro operator \(C_0,\) with generating sequence  \(\mu_n=(n+1)^{-1} ,\) and  the Euler operator \(E_\lambda,\)  with generating sequence \(\mu_n=\lambda^n,\) are  two examples of Hausdorff operators. Rhoades \cite{R} notes that  \(E_\lambda \) is bounded for \(1/2 < \lambda \leq 1.\) We show in Proposition \ref{bdd1} below that \(E_\lambda\) is  bounded also for \(0 < \lambda \leq 1/2.\) 

There is a strong connection between Hausdorff operators and the discrete Ces\`aro operator. Hurwitz and Silvermann \cite{HS} showed that the commutant of \(C_0\) is precisely the set of all Hausdorff operators, whereas Shields and Wallen~\cite{SW} showed that the commutant of \(C_0\) is the weakly closed algebra with identity generated by \(C_0.\)

\begin{proposition}
\label{perro}
If \(0 < \lambda \leq 1 \) then \(\lambda\) is an extended eigenvalue for   \(I-S^\ast,\)  and  moreover,   the Euler operator \(E_\lambda\) is a corresponding extended eigenoperator.
\end{proposition}
\begin{proof} 
First of all,  for  \(k=0\) we have
\[
E_\lambda e_0 = \sum_{n=0}^\infty (1-\lambda)^n e_n.
\]
Then, it follows from Lemma \ref{vito} that  \((I-S^\ast) E_\lambda e_0= \lambda E_\lambda (I-S^\ast )e_0.\) Next, for \(k \geq 1\) we have
\begin{align*}
S^\ast  E_\lambda e_k & =    \sum_{n=k}^\infty \binom{n}{k} \lambda^k (1-\lambda)^{n-k} e_{n-1}\\
		    & =   \lambda^k  e_{k-1} + \sum_{n=k}^\infty \binom{n+1}{k} \lambda^k (1-\lambda)^{n+1-k} e_n,
\end{align*}
so that
\begin{align*}
(I-S^\ast)   E_\lambda e_k & =  -\lambda^ke_{k-1} + \sum_{n=k}^\infty \left [  \binom{n}{k} -\binom{n+1}{k}(1-\lambda) \right ] \lambda^k (1-\lambda)^{n-k} e_n.
\end{align*}
Using Pascal's identity \(\displaystyle{\binom{n+1}{k}=\binom{n}{k}+\binom{n}{k-1}}\) leads to
\begin{align*}
(I-S^\ast)  E_\lambda e_k & =  -\lambda^ke_{k-1} + \sum_{n=k}^\infty  \left [\binom{n}{k}  \lambda   -\binom{n}{k-1} (1-\lambda)  \right ]  \lambda^k (1-\lambda)^{n-k} e_n\\
			& = 	 -\lambda^ke_{k-1} + \sum_{n=k}^\infty \binom{n}{k} \lambda^{k+1} (1-\lambda)^{n-k} e_n \\
			& -  	\sum_{n=k}^\infty  \binom{n}{k-1} \lambda^k (1-\lambda)^{n-(k-1)} e_n\\
			& = 	  \lambda \sum_{n=k}^\infty \binom{n}{k} \lambda^k (1-\lambda)^{n-k} e_n  \\
			& -  	\lambda  \left [  \lambda^{k-1}e_{k-1} + \sum_{n=k}^\infty  \binom{n}{k-1} \lambda^{k-1} (1-\lambda)^{n-(k-1)} e_n \right ]\\
			& =  	\lambda \sum_{n=k}^\infty \binom{n}{k} \lambda^k (1-\lambda)^{n-k} e_n \\
			& -  	\lambda \sum_{n=k-1}^\infty  \binom{n}{k-1} \lambda^{k-1} (1-\lambda)^{n-(k-1)} e_n  \\
			& =  	\lambda ( E_\lambda e_k- E_\lambda e_{k-1}) \\
			& = \lambda E_\lambda (I-S^\ast)e_k,
\end{align*}
so that \((I-S^\ast) E_\lambda e_k = \lambda  E_\lambda (I-S^\ast)e_k\) for all \(k \in \mathbb N,\) as we wanted.
\end{proof}

Our next goal is to describe the collection  of  the extended eigenoperators for \(I-S^\ast\) corresponding to an extended eigenvalue \(\lambda \in (0,1].\)  It is convenient  to have a digression on Toeplitz operators. We shall  follow the discussion about Toeplitz operators in  the paper of Sheldon Axler \cite{A}.

Let \((\alpha_n)_{n \in \mathbb Z} \) be a two sided  sequence of complex scalars and consider the infinite matrix \(A=(a_{nk})\)  whose entries are given by the expression
\(a_{nk} = \alpha_{n-k}.\) We say that \(A\) is  the {\em Toeplitz matrix} associated with the sequence \((\alpha_n)_{n \in \mathbb Z}.\)
The {\em Toeplitz operator} associated with a Toeplitz matrix \(A=(a_{nk})\) is defined on the complex Hilbert space \(\ell^2\) by the expression
\begin{eqnarray}
(Af)(n)= \sum_{k =0}^\infty \alpha_{n-k}f(k).
\end{eqnarray}
Consider the unit circle \(\mathbb T = \{ z \in \mathbb C \colon |z|=1\}\) and define a  function \(\varphi \colon \mathbb T \to \mathbb C\)    by the Fourier expansion 
\begin{eqnarray}
\varphi(e^{i \theta}) = \sum_{n=-\infty}^\infty \alpha_n e^{in\theta}. 
\end{eqnarray}
It is a standard fact that a Toeplitz matrix \(A\) induces a bounded operator if and only if  \(\varphi\) is essentially bounded, and moreover,
\begin{eqnarray}
\label{sup}
\|A\|= \sup \{|\varphi(z) |\colon z \in \mathbb T\}.
\end{eqnarray}

Halmos says \cite[Problem 33] {H} that Fourier expansions are formally similar to Laurent expansions,  and  the analogy motivates calling the functions of \(H^2(\mathbb T)\) the {\em analytic } elements of \(L^2(\mathbb T).\) Thus, \(\varphi\) is analytic if and only if \(\alpha_n=0\) for all \(n<0.\) Also, \(\varphi\) is called {\em co-analytic} provided that \(\alpha_n =0\) for all \(n >0.\) 

It turns out that \(AS=SA\) if and only if \(A\) is an analytic Toeplitz operator, and that \(AS^\ast=S^\ast A\) if and only \(A\) is a co-analytic Toeplitz operator.

Lambert \cite{Lam} observed that  if \(X \in {\mathcal B}(H)\) is an extended eigenoperator for an operator \(T \in {\mathcal B}(H)\) associated with an extended eigenvalue \(\lambda \in \mathbb C,\) and if \(R \in \{T\}^\prime\)   then the product \(XR\) is also an extended eigenoperator for \(T\) associated with \(\lambda.\) 

Let \(A\) be a  co-analytic Toeplitz operator. Since \(A\) commutes with \(S^\ast\)  and since \((S^\ast)^{n_0}\)  commutes with \(S^\ast,\)  it follows that \(A(S^\ast)^{n_0}\)  commutes with \(I-S^\ast.\) Since \(E_\lambda\) is an extended eigenoperator for \(I-S^\ast\) associated with the extended eigenvalue \(\lambda,\) it follows from Lambert's observation that \(E_\lambda A (S^\ast)^{n_0}\) is also an extended eigenoperator for \(I-S^\ast\) associated with the extended eigenvalue \(\lambda.\) The following result shows that these are all possible extended eigenoperators for the operator \(I-S^\ast.\)
\begin{theorem}
\label{main}
If  \(0 < \lambda \leq 1\) and  \(X\) is an extended eigenoperator of \(I-S^\ast\) associated with \(\lambda\) then there is a two sided sequence \((\alpha_n)_{n \in \mathbb Z}\)  of complex scalars  with \(\alpha_0 \neq 0\) and \(\alpha_n=0\) for all \(n \geq 1,\) and there is an  \(n_0 \in \mathbb N\) such that  \(X\) admits a factorization
\begin{eqnarray}
\label{factor}
X=E_\lambda A (S^\ast)^{n_0},
\end{eqnarray}
where \(E_\lambda\) is the Euler operator and where  \(A\) is the co-analytic Toeplitz matrix associated with  \((\alpha_n)_{n \in \mathbb Z}.\)
\end{theorem}

\begin{proof}
We have \((I-S^\ast)Xe_0=\lambda Xe_0\) and \((I-S^\ast)Xe_n=\lambda (Xe_n-Xe_{n-1})\) for all \(n \geq 1.\)  Since \(X \neq 0,\) there is some \(n \in {\mathbb N}\) such that \(Xe_n \neq 0.\) Let \(n_0=\min \{n \in \mathbb{N} \colon Xe_n \neq 0\}.\) 
 
First step: Let us suppose that \(n_0=0\) and notice that  \(Xe_0\) is an eigenvector of \(I-S^\ast\) corresponding to the eigenvalue \(\lambda,\) so that  according to Lemma \ref{vito}, there is a nonzero complex scalar \(\beta_0\) such that 
\begin{eqnarray}
\label{toro}
Xe_0  = \beta_0 \sum_{n=0}^\infty (1-\lambda)^n e_n.
\end{eqnarray}
We claim that there is a sequence of complex scalars \((\beta_n)_{n \in \mathbb N} \) with \(\beta_0 \neq 0\) and such that for every \(n \in \mathbb N,\)
\begin{eqnarray}
\label{beta}
Xe_n = \sum_{k=0}^n \beta_{n-k} E_\lambda e_k.
\end{eqnarray}
We proceed by induction. If  \(n=0,\) this follows trivially from equation~(\ref{toro}).  Then, suppose that \(n \geq 1\) and   the complex scalars \(\beta_0, \ldots , \beta_{n-1}\) are constructed in such a way that
\[
Xe_{n-1} = \sum_{k=0}^{n-1} \beta_{n-1-k} E_\lambda e_k.
\]
Notice that 
\begin{align*}
[S^\ast - (1-\lambda) I]Xe_n  & = \lambda Xe_{n-1}
 = \lambda \sum_{k=0}^{n-1} \beta_{n-1-k} E_\lambda e_k\\
& = [S^\ast - (1-\lambda) I ] \left(\sum_{k=0}^{n-1} \beta_{n-1-k} E_\lambda e_{k+1} \right),
\end{align*}
so that
\[
Xe_n - \sum_{k=0}^{n-1} \beta_{n-1-k} E_\lambda e_{k+1} \in \ker [S^\ast - (1-\lambda) I].
\]
Finally, according to Lemma \ref{vito},  there is a complex scalar \(\beta_n\) such that
\[
Xe_n - \sum_{k=0}^{n-1} \beta_{n-1-k} E_\lambda e_{k+1} = \beta_n E_\lambda e_0,
\]
and the claim follows. Now, let \((\alpha_n)_{n \in \mathbb Z}\) be the two sided sequence defined by   \(\alpha_{-n}= \beta_n\) for all \(n \geq 1\) and \(\alpha_n = 0 \)  for all \(n \in \mathbb N,\)  and let \(A\) be the co-analytic Toeplitz matrix associated with the sequence \((\alpha_n)_{n \in \mathbb Z}.\) We have that \(X=E_\lambda A,\) so that equation (\ref{factor}) holds with \(n_0=0.\)

Second step: Suppose that \(n_0 \geq 1.\)  Notice that \(Xe_n= 0\) for all  \(0 \leq n < n_0\) and \(Xe_{n_0} \neq 0.\) Thus, \((I-S^\ast)Xe_{n_0}=\lambda Xe_{n_0}\) and \( (I-S^\ast)Xe_n= \lambda (Xe_n-Xe_{n-1})
\)
for each \(n > n_0,\) or in other words,  \((I-S^\ast)XS^{n_0}e_0= \lambda XS^{n_0}e_0\) and  \((I-S^\ast)XS^{n_0} e_n= \lambda (XS^{n_0}e_n-XS^{n_0}e_{n-1})\) for each \(n \geq 1.\) This means that  \(XS^{n_0}\) is a nonzero linear operator as in the first step of the proof. Therefore, there is a sequence \((\alpha_n)\)  of complex scalars  with \(\alpha_0 \neq 0\) and such that \(XS^{n_0}=E_\lambda A,\) where \(A\) is the Toeplitz operator associated with the sequence \((\alpha_n).\) Finally, since \(Xe_n=0\) for \(0 \leq n \leq n_0\), it follows that \(X=E_\lambda A (S^\ast)^{n_0}.\)
\end{proof}
 
We finish this section with the consideration of  the question  of boundedness for the Euler operator~\(E_\lambda.\) We already mentioned   that  Rhoades \cite{R} noted that  \(E_\lambda \) is bounded for \(1/2 < \lambda \leq 1.\)  He proved  that  in fact  we have \(\|E_\lambda \| = \lambda^{-1/2}.\)  We show in Proposition~\ref{bdd1} below that \(E_\lambda\) is also bounded for \(0 < \lambda \leq 1/2\) and moreover, \(\|E_\lambda \| \leq (1-\lambda)^{-1/2}.\) Since we could not find a proof of this fact in the literature, we include an argument that is based on a criterion due to Schur. A  proof of this criterion, different from the original one, can be found in the paper  of Brown, Halmos and Shields~\cite{BHS}, where it is applied to show the boundedness of  both the continuous and the discrete Ces\`aro operators. 
\begin{lemma}[Schur test] 
\label{schur}
If \(a_{nk} \geq 0,\) if \(p_k>0,\) and if \(\alpha,\beta >0\) are such that 
\begin{align}
& \sum_{k=0}^\infty a_{nk}p_k \leq \alpha p_n,\\
& \sum_{n=0}^\infty a_{nk}p_n \leq \beta p_k,
\end{align}
then there is a bounded linear operator \(X\) with \(\|X\|^2 \leq \alpha\beta\) and such that for all \(n \in \mathbb N,\)
 \[
 (Xf)(n) = \sum_{k=0}^\infty a_{nk} f(k).
 \] 
\end{lemma}

\begin{proposition} If \(0<\lambda\leq 1/2\) then  the Euler operator \(E_\lambda\)  is bounded with \(\|E_\lambda\| \leq (1-\lambda)^{-1/2}.\)
\label{bdd1}
\end{proposition}
\begin{proof} We shall apply the Schur test to the infinite matrix
\begin{eqnarray}
a_{nk} = \left \{ \begin{array}{rl}
			\displaystyle{\binom{n}{k} \lambda^k (1-\lambda)^{n-k},} & \text{if } 0 \leq k \leq n,\\
			0, & \text{if } k>n.
			\end{array}
			\right. 
\end{eqnarray}
If we set \(p_k=1,\) then it follows from the binomial theorem that
\begin{eqnarray*}
\sum_{k=0}^\infty a_{nk}p_k & = & \sum_{k=0}^n \binom{n}{k} \lambda^k (1-\lambda)^{n-k}=1.
\end{eqnarray*}
On the other hand, using the geometric series expansion
\(\displaystyle{
(1-\lambda)^{-1} = \sum_{n=0}^\infty \lambda^n,}
\) we get 
\[
\frac{d^k}{d\lambda^k} (1-\lambda)^{-1} = \sum_{n=k}^\infty \frac{n!}{(n-k)!} (1-\lambda)^{n-k},
\]
so that
\begin{eqnarray*}
\sum_{n=0}^\infty a_{nk}p_n & = & \sum_{n=k}^\infty \binom{n}{k} \lambda^k (1-\lambda)^{n-k} = \frac{\lambda^k}{k!}  \sum_{n=k}^\infty \frac{n!}{(n-k)!} (1-\lambda)^{n-k}\\
					     & = & \frac{\lambda^k}{k!} \frac{d^k}{d\lambda^k} (1-\lambda)^{-1} = \lambda^k (1-\lambda)^{-k-1} \leq (1-\lambda)^{-1},
\end{eqnarray*}
and we  conclude that \(E_\lambda\) is bounded with \(\|E_\lambda\| \leq (1-\lambda)^{-1/2},\) as we wanted.
\end{proof}

We shall use an elementary fact that can be stated as follows.

\begin{lemma}
\label{uni}
Let \(\lambda\) be a nonzero complex number. Then \(|\lambda | + |\lambda -1| \leq 1\) if and only if \(\lambda \in (0,1].\)
\end{lemma}
\begin{proof}
Let us prove the nontrivial implication. If \(\lambda \in \mathbb R\) and \(\lambda > 1\) then we have \(|\lambda |+ | 1-\lambda | = 2\lambda -1 >1,\) and if \(\lambda \in \mathbb R\) and \(\lambda <0\) then \(|\lambda|+|1-\lambda | = 1-2\lambda >1.\)  Also, if \(\lambda \in \mathbb C\) and \({\rm Im}\,\lambda \neq 0\) then \(|\lambda | + |1-\lambda | >1\) because \(\lambda\) and \(1-\lambda\) are linearly independent over \(\mathbb R.\) 
\end{proof}

\begin{proposition} 
\label{bdd3}
If \(\lambda \in \mathbb C \backslash (0,1]\) then the Euler operator \(E_\lambda\)  is unbounded.
\end{proposition}
\begin{proof}
We have for every \(n \geq 0\)
\[
E^\ast_\lambda e_n = \sum_{k=0}^n \binom{n}{k} \lambda^k (1-\lambda)^{n-k} e_k.
\]
Using the Cauchy-Schwarz inequality gives
\begin{align*}
\|E^\ast_\lambda e_n\| & = \left \|  \sum_{k=0}^n \binom{n}{k} \lambda^k (1-\lambda)^{n-k} e_k  \right \| \\
& \geq \frac{1}{(n+1)^{1/2}}  \sum_{k=0}^n \binom{n}{k} |\lambda |^k |1-\lambda |^{n-k}\\
& =  \frac{(|\lambda |+ | 1-\lambda |)^n}{(n+1)^{1/2}} .
\end{align*}
If \(\lambda \neq 0\) then it follows from Lemma \ref{uni}  that  \(\|E^\ast_\lambda e_n\| \to \infty\) as \(n \to \infty.\) Finally, if \(\lambda =0\) then according to equation (\ref{euler}) we have \((E_0f)(n)= f(0)\) for all \(n \in \mathbb N,\) so that the  constant sequence \(E_0f\)  belongs to the complex Hilbert space  \( \ell^2\) only when  \(f(0)=0.\)
\end{proof}

\section{Extended eigenvalues for operators with rich point spectrum}
\label{sect:rich}
\noindent
We say that an operator \(T\) on a complex Banach space has {\em rich point spectrum}  provided that  \({\rm int}\, \sigma_p(T) \neq \emptyset,\) and that  for every open disc \(D \subseteq \sigma_p(T),\) the family of eigenvectors
\begin{eqnarray}
\label{total}
\bigcup_{z \in D} \ker (T-z)
\end{eqnarray}
is a total set. We  shall  see below  that  two examples of operators with rich point spectrum are the finite continuous Ces\`aro operator and  the adjoint of the discrete Ces\`aro operator. There are other natural examples  like  a bilateral weighted shift whose  point spectrum has  non empty interior, or  the adjoint of an  analytic Toeplitz operator with non constant symbol.

Recall that if \(\varphi\) is a bounded analytic function on \(\mathbb D\)  then the {\em analityc Toeplitz operator} \(T_\varphi\) is defined  on the Hardy space \(H^2(\mathbb D)\) by the expression \(T_\varphi f= \varphi \cdot f.\)
Deddens \cite{D} studied intertwining relations between analytic Toeplitz operators.  Bourdon and Shapiro~\cite{BS} generalized his work later on and they applied  it to study the extended eigenvalues of an analytic Toeplitz operator. 

Deddens showed that if there is a  non zero operator  \(X\) that intertwines  two analytic Toeplitz operators \(T_\varphi\) and  \(T_\psi,\)  that is, such that \(XT_\varphi =T_\psi X,\) then 
\begin{eqnarray}
\label{dedd}
\psi (\mathbb D) \subseteq {\rm clos}\, \varphi(\mathbb D).
\end{eqnarray}
Bourdon and Shapiro  observed that, as a consequence of this, if \(\lambda\) is an extended eigenvalue of an analytic Toeplitz operator \(T_\varphi,\) where \(\varphi\) is not constant, then   there is a non zero operator  that intertwines \( T_{\lambda \varphi}\) and \(T_\varphi,\) so that 
\begin{eqnarray}
\label{quick}
(1/\lambda)\cdot \varphi (\mathbb D) \subseteq {\rm clos}\,\varphi (\mathbb D).
\end{eqnarray}
Bourdon and Shapiro say that then  the geometry of \(\varphi(\mathbb D)\) quickly determines the extended eigenvalues of \(T_\varphi\) (for instance,    if \(\lambda\) is an extended eigenvalue of the shift operator \(T_z \in \mathcal B(H^2(\mathbb D))\)  then it follows from Deddens result that  \((1/\lambda) \cdot \mathbb D \subseteq {\rm clos}\, \mathbb D,\) and therefore \(|\lambda | \geq 1.\) )

We  prove in Theorem \ref{key} that, in general, if  an operator has rich point spectrum then the geometry of its point spectrum determines the extended eigenvalues. The precise statement of this result is provided below. Then,  we apply Theorem \ref{key}   to show that if \(\lambda\) is an extended eigenvalue for \(C_1\) on \(L^p[0,1]\)  then \(\lambda\) is real and \(0 < \lambda \leq 1\) (Corollary \ref{thm:ext}) and  if \(\lambda\) is an extended eigenvalue for \(C_0^\ast\) on \(\ell^p\) then \(\lambda\) is real and \(\lambda \geq 1\) (Corollary \ref{thm:disc}). 

As another  consequence of our general result,  in section \ref{sec:bilateral} we get that  if \(\lambda\) is an extended eigenvalue  of a bilateral weighted shift \(W\) whose point spectrum has non empty interior then \(|\lambda|=1.\) 

Finally,  if \(\lambda\) is an extended eigenvalue  of an analytic Toeplitz operator \(T_\varphi\)  on the Hardy space \(H^2(\mathbb D)\) with non constant symbol then Deddens  result (\ref{quick}) can be derived as a consequence of Theorem \ref{key}.

\begin{theorem}
\label{key}
Let us suppose that an operator \(T\) on a complex Banach space has  rich point spectrum.
If \(\lambda\) is an extended eigenvalue for \(T\) then we have 
\begin{align}
\lambda \cdot {\rm int}\, \sigma_p(T) \subseteq {\rm clos}\,\sigma_p(T).
\end{align}
\end{theorem}

\begin{proof}
Let \(X\) be an extended eigenoperator of \(T\) corresponding to the extended eigenvalue \(\lambda,\)  that is, \(X \neq 0\) and \(TX=\lambda XT.\) Let \(z \in {\rm int}\, \sigma_p(T)\) and let \(n \in \mathbb N\) such that \(D(z,1/n) \subseteq \sigma_p(T).\) Since \(X \neq 0\) and \(T\) has rich point spectrum, there exist  \(z_n \in D(z,1/n)\)  and \(f_n \in \ker (T-z_n) \backslash \{0\}\) such that  \(Xf_n \neq 0.\) Hence,
\[
TXf_n=\lambda XTf_n= \lambda z_n Xf_n,
\]
and since \(Xf_n \neq 0,\) this means that \(\lambda z_n \in \sigma_p(T).\) Taking limits as \(n \to \infty\) yields  \(\lambda z \in {\rm clos}\,\sigma_p(T),\) as we wanted. 
\end{proof}

The following result will be applied at the end of the next section to  the finite continuous Ces\`aro operator and in section \ref{sect:disc} to  the adjoint of the discrete Ces\`aro operator.
\begin{theorem}
\label{Lp}
Let \(T\) be a bounded linear operator with rich point spectrum and such that  \(\sigma_p(T)=D(r,r)\) for some \(r>0.\) If \(\lambda\) is an extended eigenvalue for \(T\) then \(\lambda\) is real and \(0 < \lambda \leq 1.\) 
\end{theorem} 

\begin{proof}
Let \(\mu=1/\lambda.\) We must show that \(\mu\) is real and \(\mu \geq 1.\)  First of all, consider the open half plane \(\Omega_r = \{ w \in \mathbb C \colon {\rm Re}\, w > 1/(2r)\},\) and notice  that   \(z \in D(r,r)\) if and only if \(1/z \in \Omega_r.\) According to Lemma~\ref{key} we have \(\mu w \in \overline{\Omega}_r\) for every \(w \in \Omega_r.\) This means that the map \(\varphi(w)=\mu w\) takes \(\Omega_r\) into \(\overline{\Omega}_r,\) and it follows from continuity that \(\varphi\) takes the closed half plane \(\overline{\Omega}_r\)  into itself. Now start with a point \(w \in \Omega_r \cap \mathbb R\) and  iterate the map \(\varphi\) to get a sequence of points \((\mu^n w)\) in \(\overline{\Omega}_r,\) so that \({\rm Re}( \mu^n w) \geq 1/(2r),\) or in other words,
\[
{\rm Re}\, \left [ \left ( \frac{\mu}{|\mu|} \right )^n \right ] \geq \frac{1}{2rw |\mu|^n} >0.
\]
Finally, write \(\mu= |\mu| (\cos \theta + i \sin \theta)\) for some \(0 \leq \theta < 2 \pi.\) Observe that \(\cos n \theta >0\) for all \(n \in \mathbb N,\) and this can only happen if \(\theta=0.\) This shows that \(\mu\) is real. It is clear that \(\mu \geq 1\) because if \(\mu <1\) then \(\varphi\) maps \(\overline{\Omega}_r\) outside \(\overline{\Omega}_r;\) for instance \(\varphi(1/(2r)) = \mu/(2r) < 1/(2r),\) and this is a contradiction.
\end{proof}

The following result will be applied in section \ref{sec:bilateral} to  a bilateral weighted shift.
\begin{theorem}
\label{annulus}
Let \(T\) be a bounded linear operator with rich point spectrum  such that for some \(0<r < R,\)
\[
\{ z \in \mathbb C \colon r < |z| < R\} \subseteq \sigma_p(T) \subseteq \{ z \in \mathbb C \colon r\leq |z| \leq R\}.
\]
If \(\lambda\) is an extended eigenvalue of \(T\) then \(|\lambda |=1.\)
\end{theorem}
\begin{proof}
Consider the region  \(\Omega = \{ z  \in \mathbb C \colon r < |z | < R\}.\) It follows from Lemma \ref{key} that the map \(\varphi (z)=\lambda z\) takes \(\Omega\) into \(\overline{\Omega},\) and it follows from continuity that \(\varphi\) maps \(\overline{\Omega}\) into itself. Start with \(z_0 \in \overline{\Omega}\) and iterate the map \(\varphi\) to obtain a sequence of points \((\lambda^n z_0)\) in \(\overline{\Omega},\) so that for all \(n \in \mathbb N\) we have
\[
r \leq |\lambda |^n \cdot |z_0| \leq R,
\] 
and notice that this can only happen if \(|\lambda|=1.\)
\end{proof}

The following result provides a sufficient condition for a general operator to have rich point spectrum. We apply that condition in the next  section  to show that  the finite continuous Ces\`aro operator \(C_1\) on \(L^p[0,1]\) has rich point spectrum. We also apply our sufficient condition  in section \ref{sect:disc} to the adjoint of the discrete Ces\`aro operator, \(C_0^\ast\) on \(\ell^q,\) and in section \ref{sec:bilateral} to a bilateral weighted shift \(W\) whose point spectrum has non empty interior. Finally, in section \ref{toeplitz} we apply a suitable modification of that condition to  the adjoint of an analytic Toeplitz operator, \(T_\varphi^\ast\) where \(\varphi\) is non constant. 

\begin{lemma}
\label{rich}
Let \(T\) be a bounded linear operator on a complex Banach space~\(E\) and let us suppose that there is  an analytic mapping \(h \colon {\rm int}\, \sigma_p(T) \to E\) with  \(h(z) \in \ker (T-z ) \backslash \{0\}\) for all \(z \in {\rm int}\, \sigma_p(T)\) and  such that  \(\{ h(z) \colon z \in  {\rm int}\, \sigma_p(T)\}\) is a total subset of \(E.\)  Then \(T\) has rich point~spectrum.
\end{lemma}
\begin{proof}
Let \(D\) be an open disc contained in \(\sigma_p(T)\) and let \(g^\ast  \in E^\ast\) such that \(\langle h(z), g^\ast  \rangle =0\) for all \(z \in D.\) We must show that  then \(g^\ast=0.\) We consider the analytic function
 \( \varphi \colon {\rm int} \, \sigma_p(T)  \to \mathbb C \) defined by \(\varphi (z) =\langle h(z), g^\ast  \rangle.\)  We have by assumption that \(\varphi\) vanishes on \(D.\) Then,  it follows from the principle of analytic continuation that \(\varphi\) vanishes on \({\rm int}\, \sigma_p(T).\) Since  the family of eigenvectors  \(\{ h(z) \colon z \in  {\rm int}\, \sigma_p(T) \}\) is a total set, it follows that \(g^\ast=0,\) as we wanted.
\end{proof}

We finish this section with a more general formulation of Theorem \ref{key} for   intertwining operators.  

\begin{theorem} 
\label{intclos} Let \(T,S\) be two  bounded  linear operators on a  complex  Banach space,  and suppose that  
there is  some \(X\) that intertwines \(T,S,\)  that is, \(X \neq 0\)   and  \(XT=SX.\) If \(T\) has rich point spectrum then  
\[
{\rm int}\, \sigma_p(T) \subseteq {\rm clos} \, \sigma_p(S).
\]
\end{theorem}  
\begin{proof}
Let \(z \in {\rm int}\, \sigma_p(T)\) and let \(n \in \mathbb N\) such that \(D(z,1/n) \subseteq \sigma_p(T).\) Since \(X \neq 0\) and since \(T\) has rich point spectrum, there exist  \(z_n \in D(z,1/n)\)  and \(f_n \in \ker (T-z) \backslash \{0\}\) such that  \(Xf_n \neq 0.\) Hence,
\[
SXf_n= XTf_n=  z_n Xf_n,
\]
and since \(Xf_n \neq 0,\) this means that \( z_n \in \sigma_p(S).\) Taking limits as \(n \to \infty\) yields  \(z \in {\rm clos}\,\sigma_p(S).\) 
\end{proof}
Notice that Theorem \ref{key} becomes a special case of Theorem \ref{intclos}  since \(\lambda\) is an extended eigenvalue for \(T\) if and only if there is some non zero  operator that intertwines \(\lambda T\) and  \(T.\)

\section{The finite continuous Ces\`aro operator on Lebesgue spaces}
\label{cesp}
Now we focus on the extended eigenvalues and extended eigenoperators for the  Ces\`aro operator \(C_1\) defined on the Lebesgue spaces \(L^p[0,1]\) for \(1 < p < \infty\) by the integral means  (\ref{cesfin}).
Leibowitz \cite{L} showed that   \(C_1\) is indeed a bounded operator on \(L^p[0,1]\) and he computed its spectrum and its point spectrum. 

\begin{theorem}
\label{thm:suffic}
If \(0<\lambda \leq 1\) then \(\lambda\) is an extended eigenvalue for the Ces\`aro operator \(C_1\) on  \(L^p[0,1]\) and a corresponding extended eigenoperator is the weighted composition operator  \(X_0 \in \mathcal B(L^p[0,1])\) defined by
\begin{eqnarray}
\label{basic}
(X_0f)(x)= x^{(1-\lambda)/\lambda} f( x^{1/\lambda}).
\end{eqnarray}
\end{theorem}

\begin{proof} First of all, let us show  that \(X_0\) is  indeed a bounded linear operator. We have for every \(f \in L^p[0,1]\)
\begin{align*}
\int_0^1 |(X_0f)(x)|^p \, dx & = \int_0^1 x^{p(1-\lambda)/\lambda} |f(x^{1/\lambda})|^p \, dx\\
& = \lambda \int_0^1 y^{(p-1)(1-\lambda)} |f(y)|^p \, dy \leq \lambda \int_0^1 |f(y)|^p \,dy,
\end{align*}
and this shows that \(X_0\) is bounded on \(L^p[0,1]\) with \(\|X_0\| \leq \lambda^{1/p}.\)

Now let us  show that \(X_0\) is an extended eigenoperator of \(C_1\) associated with the extended eigenvalue~\(\lambda.\) Let \(n \in \mathbb N\) and notice that \(X_0x^n= x^{(n+1-\lambda)/\lambda} ,\) so that 
\begin{align*}
C_1X_0x^n  = C_1  x^{(n+1-\lambda)/\lambda}  = \frac{\lambda}{n+1} x^{(n+1-\lambda)/\lambda} =  \frac{ \lambda}{n+1} X_0x^n=\lambda X_0C_1 x^n,
\end{align*}
and since the linear subspace  \({\rm span}\, \{x^n \colon n \in \mathbb N\}\) is a dense subset of \(L^p[0,1],\) it follows that \(C_1X_0=\lambda X_0C_1,\) that is, \(X_0\) is an extended eigenoperator of \(C_1\) associated with the extended eigenvalue \(\lambda.\) 
\end{proof}
Our next goal is to  show that if \(\lambda\) is an extended eigenvalue of the finite continuous Ces\`aro operator \(C_1 \in \mathcal B(  L^p[0,1)])\) then \(\lambda\) is real and \(0 <\lambda \leq 1.\) First we show that \(C_1\) has rich point spectrum. Let  \(1<p,q<\infty\) be  a pair of conjugate indices, that is, 
\[
\frac{1}{p} + \frac{1}{q}=1.
\] 
Leibowitz \cite{L} proved  the following result about the point spectrum of \(C_1.\)
\begin{lemma}
\label{pointspect}
The point spectrum of the Ces\`aro operator \(C_1\) on  \(L^p[0,1]\) is the open disc \(D(q/2,q/2).\) Moreover, each \(z \in D(q/2,q/2)\) is a simple eigenvalue of \(C_1\) and a corresponding eigenfunction is given by \(h_z (x)=x^{(1-z)/z}.\)
\end{lemma}
\noindent
The following theorem was conjectured by Borwein and Erd\'elyi~\cite{BE} and it was proven by Operstein~\cite{O}.

\begin{theorem} {\em (Full M\"untz theorem in \(L^p [0,1].\))}
\label{full}
 Let \(1 < p < \infty\) and let \((r_n)\) be a sequence of distinct real numbers greater than \(-1/p.\) Then  the linear subspace
\(
{\rm span}\, \{x^{r_0}, x^{r_1}, \ldots , x^{r_n} , \ldots \}
\)
is dense in \(L^p[0,1]\) if and only if
\begin{align}
\label{eqnmuntz}
\sum_{n=0}^\infty \frac{r_n+1/p}{(r_n+1/p)^2+1}=\infty. 
\end{align}
\end{theorem}

\begin{theorem}
\label{real}
The finite continuous Ces\`aro operator  \(C_1\) on  \(L^p[0,1]\)  has rich point spectrum.
\end{theorem}
\begin{proof}
Notice that  \(\sigma_p(C_1) = D(q/2,q/2)\) is open and connected. Also, the mapping \(h \colon \sigma_p(C_1) \to L^p[0,1]\) defined by \(h(z)(x)= x^{(1-z)/z}\) is analytic, and \(h(z) \in \ker (C_1-z) \backslash \{0\}.\) It is a standard consequence of the full M\"untz theorem  that the family of eigenfunctions \(\{h(z) \colon z \in D(q/2,q/2)\}\) is  total  in \(L^p[0,1].\)   Indeed, it suffices to consider a sequence of distinct real numbers \((z_n)\) with \(q/2< z_n < q\) and such that \(\lim z_n =q\) as \(n \to \infty,\)  since   the sequence of exponents \(r_n=(1-z_n)/z_n\) clearly satisfies  the condition (\ref{eqnmuntz}). The  result now follows from Lemma \ref{rich}.
\end{proof}
 
 \begin{corollary}
 \label{thm:ext}
If \(\lambda\) is an extended eigenvalue for  \(C_1\) on  \(L^p[0,1]\) then \(\lambda \) is real and \(0 < \lambda \leq 1.\)
\end{corollary}
\begin{proof}
This is a consequence of Theorem \ref{Lp} now that we know that \(C_1\) has rich point spectrum and that its point spectrum is the open disc \(D(q/2,q/2).\)
\end{proof}





%

\section{The discrete Ces\`aro operator on sequence spaces}
\label{sect:disc}
We  shall  prove  in this section that the set of the extended eigenvalues for the discrete Ces\`aro operator is the interval \([1,\infty)\) when \(p=2\) and that  it is contained in the interval \([1,\infty)\) when \(1<p<\infty.\)  
Let us recall that the discrete Ces\`aro operator \(C_0\) is defined on the complex Banach space \(\ell^p\)  by   the sequence of arithmetic means (\ref{cesdisc}). Rhoades \cite{R}  showed that \(C_0\) is  indeed a bounded linear operator whose point spectrum is empty and he proved the following result about the point spectrum of the adjoint operator \(C_0^\ast.\)
\begin{theorem}
\label{catch}
The point spectrum of \(C_0^\ast\) on the complex Banach space \(\ell^q\) is the open disc \(D(q/2,q/2).\) Moreover, every \(z \in D(q/2,q/2)\) is a simple eigenvalue for \(C_0^\ast\) and a corresponding eigenvector is the  sequence  \( h(z)=(h_n(z))_{n \in \mathbb N}\) defined by the relations
\begin{align}
\label{eigendisc2}
h_0(z)=1,\qquad h_n(z)= \prod_{k=1}^n \left (1 - \frac{1}{k z} \right ) \quad \text{for }n \geq 1.
\end{align}
\end{theorem}
Our first goal  is to show that if \(\lambda\) is an extended eigenvalue for \(C_0\) on \(\ell^p\)  then \(\lambda\) is real and \(\lambda \geq 1.\)  Notice that the method  that we applied  to \(C_1\) in section \ref{sect:rich} does not apply to \(C_0\)   because the point spectrum of \(C_0\) is empty. We consider instead its adjoint \(C_0^\ast.\)
\begin{theorem}
The adjoint of the discrete Ces\`aro operator \(C_0^\ast \in \mathcal B(\ell^q)\) has rich point spectrum.
\end{theorem}
\begin{proof}
Notice that  \(\sigma_p(C_0^\ast) = D(q/2,q/2)\) is open and connected. It is easy to see that the mapping \(h \colon \sigma_p(C_0^\ast) \to \ell^q\) defined by equation (\ref{eigendisc2}) is analytic, and \(h(z) \in \ker (C_0^\ast -z) \backslash \{0\}.\) It is a standard fact that the family of eigenvectors \(\{h(z) \colon z \in D(q/2,q/2)\}\) is  total  in \(\ell^q.\) As a matter of fact,   the family of eigenvectors \(\{f(1/k) \colon k \in \mathbb N\}\) is  total  in \(\ell^q,\) because \(f_n(1/k) \neq 0\) if and only if  \(n < k.\) The  result now follows at once from Lemma \ref{rich}.
\end{proof}

\begin{corollary} If \(\lambda\) is an extended eigenvalue of \(C_0\) on \(\ell^p\)  then \(\lambda \) is real and \(\lambda \geq 1.\)
\label{thm:disc}
\end{corollary}
\begin{proof}  First of all, we have \(\lambda \neq 0\) because  \(C_0\) is injective. Also, notice that  \(\lambda \)  is an extended eigenvalue for \(C_0\) if and only if \(1/ \overline{\lambda}\) is an extended eigenvalue for \(C_0^\ast,\) and therefore it is enough to show that if \(\lambda\) is an extended eigenvalue for \(C_0^\ast\) then \(\lambda\) is real and  \(0 <\lambda \leq 1.\) 
This  becomes a consequence of Theorem~\ref{Lp}  now that we know that \(C_0^\ast\) has rich point spectrum and that its point spectrum is the  disc \(D(q/2,q/2).\)  
\end{proof}
Our next goal  is to show in the Hilbertian case \(p=2\) that if \(\lambda\) is real and \(\lambda \geq 1\) then \(\lambda\) is an extended eigenvalue for~\(C_0.\)  
Kriete and Trutt \cite{KT1} showed  that \(C_0\) is subnormal using the following construction.
 Let \(\mu\) be a positive finite  measure defined on the Borel subsets of the complex plane with compact support
  and let \(H^2(\mu)\) be the closure of the polynomials on the Hilbert space \(L^2(\mu).\) Consider the shift operator \(M_z \) defined on the Hilbert space \(H^2(\mu)\) by the expression \((M_zf)(z)=zf(z).\) 
Kriete and Trutt \cite{KT1} showed that there is a   is  a positive finite  measure defined on the Borel subsets of the complex plane and supported on~\(\overline{\mathbb D},\) and there is a unitary operator \(U \colon \ell^2 \to H^2(\mu) \)  such that  
\[
I-C_0 = U^\ast M_z U,
\]
or in other words
\[
C_0= U^\ast(I-M_z)U.
\]
Then,  the extended eigenvalues of \(C_0\) are the extended eigenvalues of \(I-M_z\) and the corresponding extended eigenoperators of \(C_0\) are in one to one correspondence with the extended eigenoperators of \(I-M_z\) under conjugation with \(U,\) that is, if  a non-zero operator \(X\) satisfies \((I-M_z)X=\lambda X(I-M_z)\) then the operator \(Y=U^\ast XU\) satisfies \(C_0Y=\lambda YC_0.\)

\begin{theorem}
\label{thm:comp} If \(\lambda \geq 1\) then \(\lambda \) is an extended eigenvalue for \(I-M_z\) and a corresponding extended eigenoperator  is the composition operator \(X\) defined by the expression
\begin{eqnarray}
\label{eqn:comp}
(Xf)(z)= f \left (\frac{\lambda -1}{\lambda}+ \frac{z}{\lambda} \right ).
\end{eqnarray}
\end{theorem}
\begin{proof} Let \(f_n=Xz^n= \displaystyle{\left (\frac{\lambda -1}{\lambda}+ \frac{z}{\lambda} \right )^n}.\) We have \(\displaystyle{f_{n+1} =\left (\frac{\lambda -1}{\lambda}+ \frac{z}{\lambda} \right )f_n }\) so that
\begin{align*}
\lambda f_{n+1} & = [(\lambda -1) +M_z]f_n\\
& = \lambda f_n-(I-M_z)f_n
\end{align*}
and it follows that
\begin{align*}
 (I-M_z)f_n = \lambda(f_n-f_{n+1}) 
\end{align*}
so that 
\begin{align*}
(I-M_z)Xz^n  & = (I-M_z)f_n \\
& = \lambda(f_n-f_{n+1}) \\
& = \lambda( Xz^n -XM_zz^n)\\
& = \lambda X(I-M_z)z^n,
\end{align*}
and since the family of monomials \(\{z^n \colon n \in \mathbb N\}\) is a total set  in \(H^2(\mu),\) it follows that \((I-M_z)X=\lambda X(I-M_z).\)
\end{proof}
\begin{corollary} If \(\lambda \geq 1\) then \(\lambda\) is an extended eigenvalue for  the  discrete Ces\`aro operator \(C_0\) on  \(\ell^2.\)
\end{corollary}

\section{Extended eigenvalues for bilateral weighted shifts}
\label{sec:bilateral}
The third author \cite{P2} showed that the set of  extended eigenvalues for an injective unilateral weighted  shift is either \(\mathbb C \backslash \mathbb D\) or \(\mathbb C \backslash \{0\}.\) We consider in this section the extended eigenvalues for  a bilateral weighted shift  \(W\) on an infinite dimensional, separable complex Hilbert space \(H,\)  that is,
\begin{align}
We_n = w_n e_{n+1}, \qquad n \in \mathbb Z,
\end{align}
where \((e_n)_{n\in \mathbb Z}\) is an orthonormal basis of \(H\) and  the  sequence  \( (w_n)_{n \in \mathbb Z}\)  of non-zero weights  is bounded. 
\begin{theorem}
\label{sim}
Let us suppose that  an operator \(T\) on a complex Banach space is similar to \(\alpha T\) for some complex number \(\alpha.\) If \(\lambda\) is an extended eigenvalue for \(T\) then \(\lambda \alpha\) is an extended eigenvalue for \(T.\) 
\end{theorem}
\begin{proof}
Let \(S \) be an invertible operator such that \(\alpha T = S^{-1}TS.\) Let \(X\) be an extended eigenoperator associated with an extended eigenvalue \(\lambda\) of \(T.\) We have
\begin{eqnarray}
TX  = \lambda XT = \lambda (XS) (S^{-1}T),
\end{eqnarray}
so that 
\begin{eqnarray}
T(XS)= \lambda (XS)(S^{-1}TS)= \lambda \alpha (XS)T.
\end{eqnarray}

Notice that \(XS \neq 0\) because \(X \neq 0\) and \(S\) is onto. This means that \(\lambda \alpha\) is an extended eigenvalue for \(T\) and   \(XS\) is a corresponding extended eigenoperator.
\end{proof}

\begin{theorem}
\label{thm:bilateral}
If  \(W\) is a bilateral weighted shift then every  \(\lambda \in \mathbb T \) is an extended eigenvalue for \(W.\)
\end{theorem}
\begin{proof}
Notice that if \(W\) is a bilateral weighted shift and if \(\theta \in \mathbb R\) then \(W\) is unitarily equivalent to \(e^{i \theta} W.\) Hence, it follows from Theorem \ref{sim} with \(\alpha =e^{i\theta}\) and \(\lambda =1\) that \(e^{i \theta}\) is an extended eigenvalue for \(W.\) Thus, the unit circle \(\mathbb T = \{ \lambda \in \mathbb C \colon |\lambda |=1\}\) is contained in the set of extended eigenvalues for~\(W.\) 
\end{proof}

Shkarin \cite{S} constructed an example of  a compact, quasinilpotent bilateral weighted shift \(W\) so that the set of extended eigenvalues of \(W\) is   the unit circle. 

Now we consider the point spectrum of a bilateral weighted shift. We shall follow the discussion  in the classical survey on weighted shift operators by Allen L. Shields \cite{Shi}. Let us consider the  quantities
\begin{align}
\label{r3}
r_3^{+}(W) &  \colon= \limsup_{n \rightarrow \infty} |w_0 \cdots w_{n-1}|^{1/n}, \\ 
 \label{r2}
 r_2^{-}(W) &  \colon = \liminf_{n \rightarrow \infty} |w_{-1} \cdots w_{-n}|^{1/n}.
\end{align}
 It turns out that when  \(r_3^+(W) < r_2^-(W)\) we have
\begin{align}
 \{z \in \mathbb{C} \colon r_3^+(W) < | z | < r_2^-(W)\}  & \subseteq \sigma_p(W) \\
 \sigma_p(W) & \subseteq  \{z \in \mathbb{C} \colon r_3^+(W) \leq  | z | \leq r_2^-(W)\}.
 \end{align}
 Also,   every \(z \in \mathbb C\) with \( r_3^+(W) < | z | < r_2^-(W)\)  is a simple eigenvalue of \(W\)  and a corresponding eigenvector is given by the expression
\begin{eqnarray}
\label{eigenshift2}
h(z)= e_0 + \sum_{n=1}^\infty  \frac{w_0 \cdots w_{n-1}}{z^n}\, e_n + \sum_{n=1}^\infty \frac{z^n}{w_{-1} \cdots w_{-n}} \, e_{-n}.
\end{eqnarray}
\begin{theorem}
\label{shift2}
Let \(W\) be an injective bilateral weighted shift  on an infinite-dimensional, separable complex Hilbert space and  suppose that  \(r_3^+(W) < r_2^-(W).\) If \(\lambda\) is an extended eigenvalue for \(W\)  then \(|\lambda |=1.\)
\end{theorem}
\begin{proof}
This result becomes a consequence of Theorem \ref{annulus} if we can show that \(W\) has rich point spectrum. First of all,  the interior of the point spectrum of \(W\) is the open annulus
\begin{eqnarray}
G =  \{z \in \mathbb{C} \colon r_3^+(W) < | z | < r_2^-(W)\}. 
\end{eqnarray}
Notice that this annulus  is connected. Consider the analytic function \( h \colon G \to H\) defined by equation~(\ref{eigenshift}). We have \(h(z) \in \ker (W-z) \backslash \{0\}.\)
We must show that the family of eigenvectors \(\{ h(z) \colon  z \in G\}\) is a total subset of \(H.\)
Take any vector \(g = \sum b_n e_n \in H\) and   suppose that  \( \langle f(z),g \rangle =0\) for all \(z \in 	G.\) We ought to show that  then \(g=0.\) Consider the complex function \(\varphi \colon G \to \mathbb C\) defined by \(\varphi( z)=\langle f(z),g \rangle,\) so that
\begin{eqnarray}
\varphi (z)= \overline{b}_0 + \sum_{n=1}^\infty  \overline{b}_n \,w_0 \cdots w_{n-1} \,\frac{1}{z^n} + \sum_{n=1}^\infty \frac{\overline{b}_{-n}}{w_{-1} \cdots w_{-n}} \, z^n, \quad z \in G.
\end{eqnarray}
Thus,  \(\varphi\) is analytic and it vanishes identically  on \(G.\) Hence, \(b_n=0\) for all \(n \in \mathbb Z,\) that is, \(g=0.\)
\end{proof}

Let \(\lambda \in \mathbb T\) and let us consider the diagonal operator \(X_0= {\rm diag}\, (\lambda^{-n})_{n \in \mathbb Z}.\) We have 
\[
X_0h(z) = e_0 + \sum_{n=1}^\infty  \frac{w_0 \cdots w_{n-1}}{\lambda^n z^n}\, e_n + \sum_{n=1}^\infty \frac{\lambda^n z^n}{w_{-1} \cdots w_{-n}} \, e_{-n}=h(\lambda z),
\]
and it follows that
\[
WX_0 h(z)= Wh(\lambda z)=\lambda z h(\lambda z) = \lambda z X_0 h(z)= \lambda X_0 W h(z),
\]
and since the family of eigenvectors \(\{h(z) \colon z \in \Omega\}\) is a total set, it follows that \(WX_0=\lambda X_0W,\) so that \(X_0\) is an extended eigenoperator for \(W\) associated with the extended eigenvalue \(\lambda.\) Notice that \(X_0\) is a unitary operator since \(|\lambda|=1.\)

\section{Extended eigenvalues for analytic Toeplitz operators}
\label{toeplitz}
Now we focus on Deddens result (\ref{quick}) and we show that it can be viewed as a special case of Lemma~\ref{key}. We first show that the adjoint of a non trivial Toeplitz operator has rich point spectrum. The following result is a generalization of Lemma \ref{rich} that suits the case of the adjoint of an analytic Toeplitz operator.

\begin{lemma}
\label{rich2}
Let \(T\) be a bounded linear operator on a complex Banach space~\(E\) and suppose that there is an open connected set \(G \subseteq \mathbb C,\)   an analytic mapping \(h \colon G \to E\)  and a non constant analytic function \(\psi \colon G \to \mathbb C\) so that
\begin{enumerate}
\item  \(h(z) \in \ker [T-\psi(z)] \backslash \{0\}\) for all \(z \in G,\) and
\item  \(\{ h(z) \colon z \in G \}\) is a total set. 
\end{enumerate}
Then \(T\) has rich point spectrum.
\end{lemma}

\begin{proof}
Since \(\psi\) is a non constant function, it follows from the open mapping theorem that \(\psi (G)\) is open.  Now it follows from the first condition  that \(\psi (G)\) is contained in \(\sigma_p(T),\) so that  \({\rm int}\, \sigma_p(T) \) is non empty.
Then let \(D \subseteq \sigma_p(T)\) be an open disc, let \(G_0 = \psi^{-1}(D)\)  and let us show that the family of eigenvectors \(\{ f(z) \colon z \in G_0\}\) corresponding to eigenvalues \(\psi(z) \in D\)  is a total subset of~\(E.\) Let  \(g^\ast  \in E^\ast\) be a functional  such that \(\langle f(z), g^\ast  \rangle =0\) for all \(z \in G_0.\) We must show that then \(g^\ast=0.\) Consider the analytic function \(\varphi \colon G \to \mathbb C\) defined by \(\varphi (z) =\langle f(z), g^\ast  \rangle.\) We have by assumption that \(\varphi\) vanishes on \(G_0.\) Now it follows from the principle of analytic continuation that \(\varphi\) vanishes on \(G.\) Since  the family of eigenvectors  \(\{ f(z) \colon z \in  G \}\) is a total subset of \(E,\)  it follows that \(g^\ast=0,\) as we wanted.
\end{proof}

\begin{theorem}
\label{deddens}
If  the symbol  \(\varphi\) is not constant then the adjoint operator   \(T_\varphi^\ast\) has  rich point spectrum.
\end{theorem}
\begin{proof}
It suffices to show that \(T_\varphi^\ast\) satisfies the conditions of Lemma \ref{rich2}. Recall that the reproducing kernel \(K_z\) is the function defined for every \(z \in \mathbb D\)   by the expression
\begin{eqnarray}
K_z(w)=\frac{1}{1-\overline{z}w},
\end{eqnarray}
and it has the property that \(\langle f, K_z \rangle = f(z)\) for all \(f \in H^2(\mathbb D).\)
It is easy to see that for all  \(z \in \mathbb D\) we have
\begin{eqnarray}
T_\varphi^\ast K_z=\overline{\varphi(z)} K_z.
\end{eqnarray}
Then, consider the analytic function  \(f \colon \mathbb D \to H^2(\mathbb D)\)  defined by  \(f(z)=K_{\overline{z}}.\) We have \(T_\varphi^\ast f(z)= \overline{\varphi(\overline{z})}f(z),\) so that the first condition in Lemma \ref{rich2} is satisfied by the analytic function \(\psi(z)=\overline{\varphi(\overline{z})}.\) Moreover, it is clear that the family of   eigenvectors \(\{f(z) \colon z \in \mathbb D\}\) is a total subset of \(H^2(\mathbb D).\) 
\end{proof}

Deddens results (\ref{dedd}) and (\ref{quick})  now   follow easily.

\begin{corollary}
 If there is an operator  \(X\) that intertwines  two analytic Toeplitz operators \(T_\varphi\) and  \(T_\psi,\)  that is, such that  \(XT_\varphi =T_\psi X,\) then (\ref{dedd}) holds.
\end{corollary}
\begin{proof}
Taking adjoints yields \(T_\varphi^\ast X^\ast = X^\ast T^\ast_\psi\) with \(X^\ast \neq 0.\)  This means that \(X^\ast\) intertwines \(T_\psi^\ast\) and \(T_\varphi^\ast,\) and from Theorem \ref{intclos} we get \( {\rm int}\, \sigma_p(T_\psi^\ast) \subseteq {\rm clos}\, \sigma_p(T_\varphi^\ast).\) We have  on the one hand  \(\overline{\psi(\mathbb D)} \subseteq {\rm int}\, \sigma_p(T_\psi ^\ast)\) and  on the other hand  \({\rm clos}\, \sigma_p(T_\varphi)  \subseteq \sigma(T_\varphi^\ast) = {\rm clos}\, \overline{\varphi(\mathbb D)},
\)
so  that  \(\overline{ \psi(\mathbb D)} \subseteq {\rm clos}\, \overline{\varphi(\mathbb D)},\) as we wanted. 
\end{proof}
\begin{corollary}
If  the symbol  \(\varphi\) is not constant and if \(\lambda\) is an extended eigenvalue of \(T_\varphi\) then (\ref{quick}) holds.
\end{corollary}


\section{Factorization of extended eigenoperators in Hilbert space}
\label{sec:factor}

Now we consider  the problem of describing, for an  operator   on a complex Hilbert space,  the family of  all the extended  eigenoperators corresponding to an extended eigenvalue.

Notice that if \(X_0\) is a particular extended eigenoperator for an operator \(T\) corresponding to an extended eigenvalue \(\lambda \in \mathbb C\)  and  if \(R \in \{T\}^\prime\)  then \(X_0R\) is an extended eigenoperator for \(T\) corresponding to~\(\lambda.\) It is natural to ask whether or not all the extended eigenoperators arise in this fashion. We provide a factorization result in Theorem \ref{main7} under certain conditions  that are fulfilled by any bilateral weighted shift whose point spectrum has non-empty interior.

Our result is based on the  construction  of an analytic reproducing kernel space~\(\mathcal H\)   for an operator \(T\)  with the nice property that the shift operator \(M_z\) is bounded on \(\mathcal H\) and  that \(T^\ast\) is unitarily equivalent to the shift operator \(M_z\) on the space~\(\mathcal H.\) The construction in the particular case of the operator \(T=I-C_0^\ast\)  
appears in the paper by Shields and Wallen~\cite{SW} and also in the papers by Kriete and Trutt \cite{KT1,KT2}.

Then we apply this result to show that if \(W\) is a bilateral weighted shift whose point spectrum has non-empty interior then \(W\) has the property that every extended eigenoperator \(X\) corresponding to an extended eigenvalue \(\lambda \in \mathbb T\) factors as a product \(X=X_0R,\) where \(X_0={\rm diag}\, (\lambda^{-n})_{n \in \mathbb Z}\) is a unitary diagonal operator (a particular extended eigenoperator) and where \(R \in \{W\}^\prime.\)

We also discuss the applicability of this  result to   the finite continuous Ces\`aro operator  or the adjoint of the discrete Ces\`aro operator.

Let us recall that an {\em analytic reproducing kernel space} on an open set \(G \subseteq \mathbb C\) is a Hilbert space \(\mathcal H\) of analytic functions \( f \colon G \to \mathbb C\) such that the point evaluations \(f \mapsto f(w)\)  are bounded linear functionals.  If \(\mathcal H\) is an analytic reproducing kernel space on \(G\) then for each \(w \in G\) there exists \(K_w \in \mathcal H\) such that \(f(w)=\langle f,K_w \rangle\) for every \(f \in \mathcal H.\)
The  function \(K \colon G \times G \to \mathbb C\) 
defined  by the expression \(K(z,w)=K_w(z)\)  is called the {\em reproducing kernel} of \(\mathcal H.\) 
It follows from the reproducing property that
\[
K(z,w)= K_w(z)= \langle K_w, K_z \rangle = \overline{\langle K_z, K_w \rangle} = \overline{K_z(w)} = \overline{K(w,z)}.
\]
Since  \(K\) is analytic in \(z,\) it follows that \(K\) is  co-analytic in \(w,\) and \(K\) is said to be an {\em analytic kernel.}

If \(\varphi \colon G \to \mathbb C\) is an analytic function such that \(\varphi \cdot f \in \mathcal H\) for every \(f \in \mathcal H\) then \(\varphi\) is called a {\em multiplier.} It follows from the closed graph theorem that the  operator \(M_\varphi\) defined by \(M_\varphi f = \varphi \cdot f\) is bounded.

\begin{theorem}
\label{main7} 
Let \(T\) be an operator on a complex Hilbert space \(H,\)  let \(G \subseteq  \mathbb C\) be an  open connected set and  suppose that 
there is an analytic mapping \(h \colon G \to H\) such that 
\begin{enumerate}
\item[{\rm (i)}] \(\dim \ker (T -z)=1\) for every \(z \in G,\)
\item[{\rm (ii)}]  \(h(z) \in \ker (T- z) \backslash \{0\}\) for every \(z \in G,\) 
\item[{\rm (iii)}] \(\{h(z) \colon z \in G  \}\) is a total subset of \(H.\)  
\end{enumerate}
Then there exists an analytic  reproducing kernel  space \(\mathcal H\) on \(G\) with the property that \(M_z\) is bounded on~\(\mathcal H,\) and there exists a unitary operator \(U \colon H \to \mathcal H\) such that \(T^\ast =U^\ast M_z U.\) 
\end{theorem}

\begin{proof}
Let \(f \in H\) and let \(\hat{f} \colon G \to \mathbb C\) be the analytic function defined by the expression  \(\hat{f}(z)= \langle f, h(\overline{z})  \rangle.\) Let \(\mathcal H\) be the Hilbert space of all functions \(\hat{f}\) provided with the norm \(\|\hat{f}\|=\|f\|.\) It is clear that the map \(U \colon H \to \mathcal H\) defined by \(Uf=\hat{f}\) is a unitary operator and that for every \(z \in G\) we have
\begin{align*}
(UT^\ast f)(z) & = \langle T^\ast f, h(\overline{z}) \rangle\\
& = \langle f, T h(\overline{z}) \rangle\\
& = \langle f, \overline{z}h(\overline{z}) \rangle\\
& = \langle z f,h(\overline{z}) \rangle\\
& = (M_z Uf)(z).
\end{align*}
It follows that \(UT^\ast =M_zU,\) so that \(M_z\) is bounded on \(\mathcal H,\) and  \(T^\ast  = U^\ast M_z U.\) 
\end{proof}

The following result about multipliers is an important  tool for  the proof of Theorem \ref{model}. It is stated as  Lemma 5 in the paper of Shields and Wallen \cite{SW}.
\begin{lemma}
\label{mult} If \(\varphi \in H^\infty(G)\) then the multiplication  operator \(M_\varphi\) defined by \(M_\varphi f = \varphi \cdot f\) is  a bounded   linear operator on~\(\mathcal H\) with \(\|M_\varphi\|=\|\varphi\|_\infty.\) 
\end{lemma}

Another tool for the proof of Theorem \ref{model} is a result that has been extracted with slight modifications from the proof of the main theorem in the paper by Gonz\' alez and the second author~\cite{GLS}.
\begin{lemma} \label{comm}  Let  \(T \in \mathcal B(H)\) be an operator as in Theorem \ref{main7} and let  \(X \in \mathcal B(H).\) The following  are equivalent:
\begin{enumerate}
\item[(a)] \(TX=XT,\) 
\item[(b)]  there is a bounded analytic function \(\varphi \colon G \to \mathbb C\) such that  for all \(z \in G,\)  
\begin{align}
\label{psi}
Xh(z)=\varphi (z) h(z).
\end{align}
\end{enumerate}
\end{lemma}

\begin{proof} First of all, if \(TX=XT\) then \(TXh(z)=XTh(z)=z Xh(z),\) so that \(Xh(z) \in  \ker (T- z)\) and it follows from (i) that  there is a  function \(\varphi \colon G \to \mathbb C\) such that
\(Xh(z)=\varphi (z) h(z).\) We claim that \(\varphi\) is analytic.  Let \(z_0 \in G\) and let \(g \in H \backslash \{0\}\) such that \(\langle h(z_0),g^\ast \rangle \neq 0.\)  Then we have
\begin{align}
\varphi(z)= \frac{\langle Xh(z), g \rangle}{\langle h(z), g \rangle},
\end{align}
so that \(\varphi\) is analytic at \(z_0\)  because it is the quotient of two analytic functions where the denominator does not vanish in a neighborhood of  \(z_0.\)
Also, it is clear that \(\varphi\) is bounded with \(\|\varphi\|_\infty \leq \|X\|.\) Conversely,  suppose   (b) holds. We have
\begin{align*}
TXh(z)& =\varphi(z)Th(z)\\
& = z \varphi(z)h(z)\\
& = z Xh(z) \\
& = XTh(z).
\end{align*}
Finally, it follows from (iii) that \(TX=XT.\)
\end{proof}

The next  result is the key   to the factorization of an extended eigenoperator.
\begin{lemma}
\label{extend}
Let \(T\) be an operator as in Theorem \ref{main7} and let  \(\lambda\) be an extended eigenvalue  of  \(T.\) Let us suppose that \(\lambda\) satisfies 
\(\lambda \cdot G \subseteq G\) and let \(X\) be a corresponding extended eigenoperator. Then there exists an analytic function  \(\varphi \colon G \to \mathbb C\) such that  for all \(z \in G\) we have 
\begin{align}
\label{factor7}
Xh(z)=\varphi (z) h(\lambda z).
\end{align}
\end{lemma}
\begin{proof} First of all, since \(X\) is an extended eigenoperator corresponding to \(\lambda\) and since \(h(z)\) is an eigenvector corresponding to \(z,\) we get
\begin{align*}
TXh(z)=\lambda XTh(z)= \lambda z Xh(z)
\end{align*}
 for every \(z \in G.\) This means that  \(Xh(z) \in \ker (T-\lambda z),\) and  it follows from (i) that there is a function \(\varphi \colon G  \to \mathbb C\) such that \(Xh(z)=\varphi(z) h(\lambda z).\) 
We claim that \(\varphi\) is analytic. Indeed, let \(z_0 \in G\) and let \(g \in H\) such that \(\langle f(\lambda z_0), g \rangle \neq 0.\) Then 
\begin{align}
\varphi(z) = \frac{\langle Xh(z), g \rangle}{\langle h(\lambda z),g \rangle},
\end{align}
so that \(\varphi\) is analytic  at \(z_0\)  because it is the quotient of two analytic functions where the denominator does not vanish in a neighborhood of \(z_0.\)
\end{proof} 

We say that an analytic reproducing kernel space \({\mathcal H}\) is {\em dilation invariant} provided that, for every \(\lambda \in \mathbb C\) such that \(\lambda G \subseteq G,\) the composition operator \(Y_0\) defined by the expression
\begin{align}
\label{basic2}
(Y_0\hat{f})(z)= \hat{f}(\lambda z).
\end{align}
is a bounded linear operator on \(\mathcal H.\)
\begin{lemma}
\label{eigen}
Let us suppose that  the model space \({\mathcal H}\)  of Theorem \ref{main7}  is dilation invariant,  let \(\lambda\) be a complex scalar such that \(\lambda G \subseteq G,\) let \(Y_0\) be the composition operator defined on \({\mathcal H}\) by equation (\ref{basic2}), and set \(X_0=U^\ast Y_0U.\) Then \(\lambda\) is an extended eigenvalue for \(T\) and  \(X_0\) is a corresponding  extended eigenoperator. 
\end{lemma}
\begin{proof} We claim that \(X_0h(z)=h(\lambda z)\) for every \(z \in G.\) The result then follows easily because
\begin{align*}
TX_0h(z) & =Th(\lambda z)\\
& =\lambda z h(\lambda z)\\
& = \lambda z X_0h(z)\\
& = \lambda X_0Th(z),
\end{align*}
and from (iii) we get \(TX_0=\lambda X_0T.\)
Now, for the proof of our claim,  observe that \(UX_0=Y_0U,\) so that
\(
UX_0h(z)  = Y_0Uh(z)
= Uh(\lambda z),
\)
and the claim follows.
\end{proof}

\begin{theorem} 
\label{model}
Suppose that  the model space \({\mathcal H}\)  of Theorem \ref{main7}  is dilation invariant and that the extended eigenoperator \(X_0\) of Theorem \ref{eigen} is bounded below, i.e., there is a constant \(c >0\) such that \(\|X_0f \| \geq c \|f\|.\) If \(X\) is an extended eigenoperator for \(T\) corresponding to \(\lambda\) then there exists \(R \in \{T\}^\prime\) such that \(X=X_0R.\)
\end{theorem}
\begin{proof}
First of all, apply Lemma  \ref{extend} to find an analytic function  \(\varphi \colon G \to \mathbb C\) such that  for all \(z\in G,\)
\begin{align}
\label{psi3}
Xh(z)=\varphi (z) h(\lambda z).
\end{align}
Notice that \(Xh(z)=\varphi(z) X_0h(z),\) and since \(X_0\) is bounded below, we get
\[
|\varphi(z)| = \frac{\|Xh(z)\|}{\|X_0h(z)\|} \leq \frac{1}{c} \cdot \frac{\|Xh(z)\|}{\|h(z)\|} \leq \frac{1}{c} \cdot \|X\|,
\]
so that \(\varphi\) is bounded.  Then, consider the  analytic function \(\psi(z)=\overline{\varphi(\overline{z})}.\)  Thus,  \(\psi \in H^\infty(G),\) and according to Lemma \ref{psi}, the multiplication  operator \(M_\psi\) defined by \(M_\psi f = \psi \cdot f\) is  a bounded   linear operator  on~\(\mathcal H.\)  Next, consider the operator \(R=U^\ast M_\psi^\ast  U.\) We claim  that for all \(z \in G\) we have
\begin{align}
\label{joy}
Rh(z) & =\varphi(z)h(z).
\end{align}
Indeed, from the definition of \(R\) we have
\[
URh(z)= M_\psi^\ast Uh(z),
\]
so that for all \(z, \xi \in G\) we get
\begin{align*}
[URh(z)](\xi) & = [M_\psi^\ast Uh(z)] (\xi)  \\
& =  \langle M_\psi^\ast U h(z),Uh(\overline{\xi}) \rangle \\
& = \langle  U h(z), M_\psi Uh(\overline{\xi})\rangle \\
& =  \overline{\langle  M_\psi Uh(\overline{\xi}), U h(z)  \rangle} \\
& = \overline{[M_\psi Uh(\overline{\xi})](\overline{z})}\\
& = \overline{\psi(\overline{z})} \cdot \overline{ [ Uh(\overline{\xi})](\overline{z})}\\
& = \varphi(z) \cdot \overline{\langle Uh(\overline{\xi}), Uh(z) \rangle}\\
& =  \varphi(z)  \cdot \langle Uh(z), Uh(\overline{\xi}) \rangle\\
& = \varphi(z)  \cdot [Uh(z)](\xi),
\end{align*}
so that \(URh(z)=\varphi(z) Uh(z)\) for all \(z \in G\) and the claim follows. Finally, it follows from equation (\ref{joy}) and Lemma \ref{comm} that \(R \in \{T\}^\prime.\) Moreover,
\(
Xh(z)=\varphi(z)X_0h(z)= X_0Rh(z)
\)
for all \(z \in G,\) and it follows from (iii) that \(X=X_0R,\) as we wanted.
\end{proof}

Let \(W\) be an injective  bilateral weighted shift   on an infinite-dimensional, separable complex Hilbert space \(H,\)  so that for every \(n \in \mathbb Z\) we have 
\begin{align}
We_n = w_n e_{n+1},
\end{align}
where \((e_n)_{n \in \mathbb Z}\) is an orthonormal basis of \(H\) and  the  sequence  \( (w_n)_{n \in \mathbb Z}\)  of non-zero weights  is bounded. 
Recall that the point spectrum of \(W\) is the open annulus \( G=\{z \in \mathbb{C} \colon r_3^+(W) < | z | < r_2^-(W)\}.\) Also, recall  that   every \(z \in G\)  is a simple eigenvalue of \(W\)  and a corresponding eigenvector is given by 
\begin{align}
\label{eigenshift}
h(z)= e_0 + \sum_{n=1}^\infty  \frac{w_0 \cdots w_{n-1}}{z^n}\, e_n + \sum_{n=1}^\infty \frac{z^n}{w_{-1} \cdots w_{-n}} \, e_{-n}.
\end{align}

It is easy to see that conditions (i), (ii) and (iii) of Theorem \ref{main7} are satisfied. Then,
let \(\lambda \in \mathbb T\) and consider the unitary diagonal operator \(X_0={\rm diag}\, (\lambda^{-n})_{n \in \mathbb Z}.\) A direct computation shows that that   \(X_0\) is an extended eigenoperator for \(W\) corresponding to the extended eigenvalue \(\lambda,\) and moreover, \(X_0h(z)=h(\lambda z).\) Therefore, the model space \(\mathcal H\) of Theorem \ref{main} is dilation invariant, and the operator \(X_0\) is bounded below. Thus, we get the following

\begin{corollary}
\label{shift}
Let \(W\) be an injective bilateral weighted shift  on an infinite dimensional, separable complex Hilbert space and   suppose that  \(r_3^+(W) < r_2^-(W).\) Let   \(X\) be an extended eigenoperator for \(W\)  correponding to some extended eigenvalue \(\lambda \in \mathbb T.\) Then \(X\) admits a factorization 
\[
X=X_0 R,
\] 
where \(X_0={\rm diag}\, (\lambda^{-n})_{n \in \mathbb Z}\) is a unitary diagonal operator (a particular  extended eigenoperator for \(T\)) and where \(R \in \{W\}^\prime.\)
\end{corollary}

Let us see if Theorem \ref{main7} can be applied to  \(C_1.\) Let \(G= \{ z \in \mathbb C \colon |z-1|<1\}\) and let \(h \colon G \to L^2[0,1]\) be the analytic mapping defined by the expression 
\begin{align}
h(z)(x)=x^{(1-z)/z}.
\end{align} 
We have already seen that the conditions (i), (ii) and (iii) of Theorem \ref{main7} are fulfilled. Then, let \(0 < \lambda \leq 1\) and consider the weighted composition operator \(X_0\) defined on \(L^2[0,1]\) by the expression
\[
(X_0f)(x)=x^{(1-\lambda)/\lambda} f(x^{1/\lambda}).
\]
We know that \(X_0\) is bounded with \(\|X_0\| \leq \lambda^{1/2}\) and that  \(X_0h(z)=h(\lambda z).\) It follows  that the model space \(\mathcal H\) is dilation invariant. However, we cannot apply Theorem \ref{model} because \(X_0\) is not bounded below. Indeed, if \(X_0\) is bounded below then there is  a constant \(c>0\) such that  \(\| X_0f\| \geq c \|f\|\) for all \(f \in L^2[0,1],\) so that 
\begin{align*}
\frac{1}{c^2} &  \geq \frac{\|h(z)\|_2^2} {\|X_0h(z)\|_2^2} \\
& = \frac{\|f(z)\|_2^2}{\|f(\lambda z)\|_2^2}\\
& = \frac{\displaystyle{2\, {\rm Re}\, \frac{1-\lambda z}{\lambda z} +1}}{\displaystyle{2\, {\rm Re}\, \frac{1-z}{z} +1}} \\
& =  \frac{\displaystyle{2 \left ( \frac{1-\lambda}{\lambda}+ \frac{1}{\lambda} {\rm Re}\, \frac{1-z}{z} \right ) +1}}{\displaystyle{2\, {\rm Re}\, \frac{1-z}{z} +1}} \to \infty \quad \text{as } z \to 2,
\end{align*}
and this is a contradiction.

Let us see if Theorem \ref{main7} can be applied to the adjoint of the discrete Ces\`aro operator. We consider the operator \(T=VC_0^\ast V^\ast \in \mathcal B(H^2(\mathbb D))\) and the analytic mapping \(h \colon G \to H^2(\mathbb D)\) defined by the expression
\(
h(z)= Vg(z),
\)
so that 
\(
h(z)(\xi) = (1-\xi)^{(1-z)/z}.
\)
It is easy to see that \(h\) is analytic on \(G\) and  that  the conditions (i), (ii) and (iii) of Theorem~\ref{main7} are satisfied. However, we cannot apply Theorem~\ref{eigen} because the model space \(\mathcal H\) fails to be dilation invariant. Indeed, if \({\mathcal H}\) is dilation invariant then for every \(0 < \lambda <1\) there is a constant \(c>0\) such that \(\|h(\lambda z) \| \leq c \|h(z)\|.\) When  \(\lambda =1/2,\) we set \(z=1/(n+1)\) and   we get
\begin{align*}
h (z)(\xi) & = (1-\xi)^n,\\
h (z/2) (\xi) & =  (1-\xi)^{2n+1},
\end{align*}
so that for every \(n \in \mathbb N\) we have
\[
\|(1-\xi)^{2n+1}\|_{H^2(\mathbb D)}^2 \leq c^2 \|(1-\xi)^{n}\|_{H^2(\mathbb D)}^2. 
\]
Use the binomial theorem to get
\[
(1-\xi)^{n} = \sum_{k=0}^n (-1)^k \binom{n}{k} \xi^k.
\]
It follows from  Parseval's identity  that
\[
\|(1-\xi)^n\|_{H^2(\mathbb D)}^2 = \sum_{k=0}^n \binom{n}{k}^2 = \binom{2n}{n}.
\]
Then we have
\begin{align*}
c^2  & \geq \frac{\|(1-\xi)^{2n+1}\|_{H^2(\mathbb D)}^2}{\|(1-\xi)^n\|_{H^2(\mathbb D)}^2}\\
& =\frac{ \displaystyle{\binom{4n+2}{2n+1}}}{ \displaystyle{\binom{2n}{n}}}\\
& = \frac{(4n+2)! \,n! \, n!}{(2n+1)! \, (2n+1)! \, (2n)!}, \\ 
\end{align*}
but  using Stirling's formula, the last expression is approximately  \( 2^{2n+2},\) and this is   a contradiction.

\section{The infinite continuous  Ces\`aro operator on Hilbert space}
\label{infinite}
As we mentioned in the introduction,  in this section we show   that, in contrast with the operator \(C_1,\)  the set of extended eigenvalues for the operator \(C_\infty\) is as small as it can be, that is, it reduces to \(\{1\}.\)

There are several examples of Hilbert space operators with this property in the literature. It is worth mentioning some of them. Biswas and the third author~\cite{BP} showed that if \(Q \in \mathcal B(H)\) is a quasinilpotent operator then the set of extended eigenvalues for \(\alpha +Q\) for every complex number \(\alpha \neq 0\) reduces to~\(\{1\}.\) They also showed when   \(\dim H < \infty\) that   the set of extended eigenvalues for  \(T \in \mathcal B(H)\) reduces to \(\{1\}\) if and only if  \(\sigma(T)=\{\alpha\}\) for some complex number \(\alpha \neq 0.\) Finally,   an example was given by Shkarin~\cite{S}  of a compact quasinilpotent operator on a Hilbert space  whose  set of extended eigenvalues reduces to \(\{1\},\) answering at once  two questions raised by Biswas, Lambert and the third author~\cite{BLP}. 

Brown, Halmos and Shields \cite{BHS} proved that \(C_\infty\) is indeed a bounded linear operator, and they also proved that \(I-C_\infty^\ast\) is unitarily equivalent to a bilateral  shift of multiplicity one.  

Recall that a bounded linear operator \(U\) on a complex Hilbert space \(H\)  is a {\em bilateral  shift of multiplicity one} provided that  there is an orthonormal basis \((e_n)\) of \(H\)  such that   \(Ue_n=e_{n+1}\) for all \(n \in \mathbb Z.\)  

Consider a bilateral shift of multiplicity one \(U \in {\mathcal B}(L^2[0,1])\) and  a unitary operator  \(V \in {\mathcal B}(L^2[0,1])\)  such that \(I-C_\infty ^\ast=V^\ast UV.\) We have \[C_\infty=V^\ast (I-U^\ast)V,\] and  it follows that the extended eigenvalues of \(C_\infty\) are precisely the extended eigenvalues of \(I-U^\ast,\) and that  the extended eigenoperators of \(C_\infty\) are in one to one correspondence with   the extended eigenoperators of \(I-U^\ast\)  under conjugation with  \(V.\)
	
\begin{lemma}\label{lemma3jsp}
Let $X$ be an operator satisfying $(I-U^\ast)X=\lambda X(I-U^\ast)$, and let $\dots X_{-1},X_0,X_1,X_2,\dots$ be the rows of the matrix of $X$. Then $$X_{n+1}=\left(\lambda U+1-\lambda\right) X_n,$$ for all $n\in\mathbb{Z}$. Consequently, for any $m,n\in\mathbb{N}$, $$X_{m+n}=\left( \lambda U+1-\lambda\right)^{n} X_m.$$ In particular, if $m=0$, $X_n=\left( \lambda U+1-\lambda\right)^{n} X_0$, for all $n\in\mathbb{N}$.
\end{lemma} 
\begin{proof} Taking adjoints we obtain \(X^\ast (I-U) = \overline{\lambda} (I-U)X^\ast\) so that \(X^\ast e_n -X^\ast e_{n+1}= \overline{\lambda} (I-U)X^\ast e_n\) and therefore \(X^\ast e_{n+1}=(\overline{\lambda} U + 1-\overline{\lambda})X^\ast e_n.\) Hence, \(X_{n+1}= \overline{X^\ast e_{n+1}}= (\lambda U + 1-\lambda) \overline{X^\ast e_n}= (\lambda U + 1-\lambda) X_n.\)
\end{proof}

\begin{theorem}\label{thm4.3}
Let $U$ be a bilateral shift of multiplicity one, and let $\lambda$ be a complex number with $\lambda \neq 1.$ Then the equation $(I-U^*)X=\lambda X(I-U^*)$ has  only the trivial solution $X=0$. 
\end{theorem} 
\begin{proof}
Let $A$ be a subset of the interval $[0,2\pi)$ such that $|\lambda e^{it}+1-\lambda|>1$ for all $t\in A$. Each row $X_n$ of the matrix for $X$ is a doubly infinite, square summable sequence of complex numbers, so it can be identified with a function in $L^2(\mathbb{T})$, with these complex numbers as its Fourier coefficients. Since every point on the unit circle is of the form $e^{it}$ for a unique $t\in [0,2\pi)$, the set $A$ corresponds to a subset $A'$ of $\mathbb{T}$. We will show that $X_0$ is equal to $0$ almost everywhere on $A'$. Indeed, if that was not the case, there would exist a set $A_0\subset A$ of positive measure and a constant $c>0$ such that $|X_0(t)|\ge c$ and $|\lambda e^{it}+1-\lambda|\ge 1+c$ for all $t\in A_0$. It would then follow that for every $n\in\mathbb{N}$,
\begin{align*}
\|X_n\|^2 &= \int_0^{2\pi} |X_n(t)|^2\,dt \\
& = \int_0^{2\pi} |\left( \lambda U+1-\lambda\right)^{n} X_0(t)|^2\,dt \\
& = \int_0^{2\pi} |\left( \lambda e^{it}+1-\lambda\right)^{n}|^2\,| X_0(t)|^2\,dt \\
& \ge \int_{A_0} |\left( \lambda e^{it}+1-\lambda\right)^{n}|^2 |X_0(t)|^2\,dt \\
& \ge \int_{A_0} (1+c)^{2n}c^2\,dt \to \infty, \mbox{ as }n\to\infty.
\end{align*}
Now we turn our attention to the set $B\subset [0,2\pi)$ such that $|\lambda e^{it}+1-\lambda|<1$ for all $t\in B$. Once again, $X_0$ is equal to $0$ for almost every $t\in B$. Otherwise, there would be a set $B_0\subset B$ of positive measure and a constant $d\in(0,1)$ such that $|X_0(t)|\ge d$ and $d\le |\lambda e^{it}+1-\lambda|\le 1-d$ for all $t\in B_0$. It would then follow that for every negative integer $n$,
\begin{eqnarray*}
\|X_n\|^2  & = \int_0^{2\pi} |X_n(t)|^2\,dt \ge \int_{B_0}  |X_n(t)|^2\,dt = \int_{B_0} |X_0(t)|^2 |\lambda e^{it}+1-\lambda|^{2n}\,dt \\
& \ge \int_{B_0} d^2(1-d)^{2n} \,dt  \to \infty, \mbox{ as }n\to -\infty.
\end{eqnarray*}
Thus, the function $X_0$ is zero almost everywhere on $A\cup B$. The complement of this set in $[0,2\pi)$ consists of two points. These are the points of intersection of the unit circle and the circle with center $(1-\lambda)/\lambda$ and radius $1/|\lambda|$. The only exceptions occurs when $\lambda=1$ and $\lambda=0$. In the former case, the two circles coincide, and in the latter $|\lambda e^{it}+1-\lambda|=1$ for all $t\in [0,2\pi)$. However, the case $\lambda=0$ has been ruled out since the kernel of $I-U^\ast$ is trivial. 

We conclude that, unless $\lambda=1$, $X_0$ is the zero function in $L^2([0,2\pi))$ and, by Lemma~\ref{lemma3jsp}, the same is true of $X_n$ for any $n\in\mathbb{Z}$. Consequently, $X=0$ and the theorem is proved.
\end{proof}

 \section{The infinite continuous  Ces\`aro operator on Lebesgue spaces}
 \label{sec:cinftyp}
  Let \(1 <p,q<\infty\) be conjugate indices, that is, 
 \[
 \frac{1}{p}+\frac{1}{q}=1.
 \]

 Our aim in this section  is to show that the set of extended eigenvalues for the infinite continuous Ces\`aro operator \(C_\infty\) on the complex Banach space \(L^p[0,\infty)\) reduces to the singleton \(\{1\}.\)

Before we present our  result we  define a sequence of functions $\{e_n\}_{n\in\mathbb{Z}}$ in $L^q(0,\infty)$. This construction is modeled after the one in \cite{BHS} for the case $q=2$. Let  $e_0=\chi_{(0,1)},$ and  let $$e_n=(1-2/q\,C_\infty^*)^n e_0, \quad \text{for } n\in\mathbb{N}.$$ Next, we define an operator $R$ on the linear span of $\{e_n\}_{n\in\mathbb{N}}$ by 
 $$
 Rf(x) = 
 -x^{-2/q} 
 f\left(\frac{1}{x}\right),
 $$
 and define  $e_{-n}=Re_{n-1}(x),$ for $n\in\mathbb{N}.$ 
 \begin{proposition}\label{propjsp101}
 Let the sequence of functions $\{e_n\}_{n\in\mathbb{Z}}$ be defined as above. Then $\{e_n\}_{n\in\mathbb{Z}}$ is a linearly independent set of functions in $L^q(0,\infty)$ and its closed linear span is $L^q(0,\infty)$. Further, the operator $1-2/q\,C_\infty^*$ shifts this sequence, i.e., $(1-2/q\,C_\infty^*)e_n=e_{n+1}$ for all $n\in\mathbb{Z}$. Finally, for any $\gamma\in (0, 1)$  there exists  $K=K(\gamma)$ such that $\|e_n\|\le K \gamma^{-n}$ if $n\ge 0$, and $\|e_n\|\le K \gamma^{n}$ if $n<0$. 
 \end{proposition}
 \begin{proof}
 We start with the observation that the Cesaro operator $C_\infty$ is a bounded operator on $L^p(0,\infty)$, so its adjoint $C_\infty^*$ is bounded on $L^q(0,\infty)$. Therefore, $e_n\in L^q(0,\infty)$ for $n\ge 0$. Furthermore, it is straightforward to verify that $\|Re_n\|_q =  \|e_n\|_q$, so $e_n \in L^q(0,\infty)$ for $n< 0$ as well.
 
 Next we will show that $\{e_n\}_{n\in\mathbb{Z}}$ is a total set in $L^q(0,\infty)$. First we notice that for $n\ge 0$, each function $e_n$ vanishes outside $[0,1]$, and for $n<0$ outside of $(1,+\infty)$. In both cases it suffices to demonstrate that if a bounded linear functional vanishes on all $\{e_n\}$ then it must be the zero functional. Further, each functional on $L^q(0,1)$ can be represented by a function $g\in L^p(0,1)$. So, suppose that $g$ is such a function and that $\int_0^1 e_n \overline{g}=0$ for all $n\ge 0$.
  Let $g_n=(I-C_\infty^*)^n e_0$, for $n\ge 0$. It was proved in \cite{BHS} that $\{g_n\}$ is an orthonormal system in $L^2(0,1)$. Further,
  $$
  g_n=\left[1-\frac{q}2+\frac{q}2\left(1-\frac2{q}C_\infty^*\right)\right]^n e_0 = \sum_{i=0}^n \binom ni \left(1-\frac{q}2\right)^{n-i}\left(\frac{q}2\right)^i e_i,
  $$
  so $\int_0^1 g_n \overline{g}=0$ for all $n\ge 0$. Thus, any bounded linear functional that vanishes on $\{e_n\}$ must vanish on $\{g_n\}$, hence on $L^2(0,1)$, and it must be zero. When $n<0$, we will assume that $g\in L^p(1,\infty)$ and that $\int_1^\infty e_{-n} \overline{g}=0$ for all $n\ge 1$. However, using the substitution $t=1/x$,
  \begin{align*}
  \int_1^\infty e_{-n}(x) \overline{g}(x)\,dx  & = 
  -\int_1^\infty x^{-2/q} e_{n-1}(1/x) \overline{g}(x)\,dx \\
  & =  -\int_0^1 t^{-2/p} e_{n-1}(t) \overline{g}(1/t)\,dt.
 \end{align*}
  So, the previous case implies that $ t^{-2/p}\overline{g}(1/t)$ is the zero function, whence $g=0$.

 Next we consider the set $\mathcal{F}$ defined as follows. A function $f\in L^q(0,\infty)$ belongs to $\mathcal{F}$ if there exists a sequence of complex numbers $\{c_n\}_{n\in\mathbb{Z}}$ such that $f=\sum_{n\in\mathbb{Z}} c_ne_n$. Since  $\{e_n\}_{n\in\mathbb{Z}}$ is a total set, $\mathcal{F}$ is dense in $L^q(0,\infty)$. Now we will show that if $f\in \mathcal{F}$, there is exactly one sequence $\{c_n\}_{n\in\mathbb{Z}}$. In order to do that it suffices to demonstrate that, if $\sum_{k\in\mathbb{Z}} c_ke_k=0$ then $c_k=0$ for all $k\in\mathbb{Z}$. Notice that
 \begin{align*}
  \left \|\sum_{k\in\mathbb{Z}} c_ke_k \right \|^q & =
 \int_0^\infty \left|\sum_{k\in\mathbb{Z}} c_ke_k\right|^q = \int_0^1 \left|\sum_{k=0}^\infty c_ke_k\right|^q +
 \int_1^\infty \left|\sum_{k=-\infty}^{-1} c_ke_k\right|^q \\
 & =   \left \|\sum_{k=-\infty}^{-1} c_ke_k \right \|^q + \left \|\sum_{k=0}^\infty c_ke_k \right \|^q,
\end{align*}
 so we can consider separately $n\ge 0$ and $n<0$. We start with $n\ge 0$. 
 Let $\alpha \in D(q/2,q/2)$ and 
 \(f_\alpha (x)=x^{(1-\alpha)/\alpha}\).
 Since $\|f\|_q \ge |\int_0^1 f \overline{f_\alpha}|/\|f_\alpha\|_p$ for any $f\in L^q(0,1)$ and $f_\alpha\in L^p(0,1)$ it follows that
 $$
\int_0^1 \left(\sum_{k=0}^\infty c_ke_k\right) \overline{f_\alpha}= 0.
 $$
 Notice that, if $k\ge 0$
 $$
 \int_0^1 e_k \overline{f_\alpha} = \int_0^1 \left(1-\frac{2}{q}C_\infty^*\right)^ke_0 \overline{f_\alpha} = \int_0^1 e_0 \left(1-\frac{2}{q}C_\infty^k\right) \overline{f_\alpha}.
 $$
 Further, $(1-2/q\,C_\infty)^k \overline{f_\alpha} = (1-2/q\,\overline{\alpha})^k \overline{f_\alpha}+v_k$, where $v_k$ is a function that vanishes on $(0,1)$. Thus,
 $$
 \int_0^1 \sum_{k=0}^\infty c_k\left(1-\frac{2}{q}\alpha\right)^k   \overline{f_\alpha}= 0 .
 $$
 It is easy to see that $\int_0^1 \overline{f_\alpha}\ne 0$, so we obtain that 
 $$
\sum_{k=0}^\infty c_k\left(1-\frac{2}{q}\alpha\right)^k = 0.
 $$
 This implies that the analytic function $\sum_{k=0}^\infty  (1-2z/q)^k$ vanishes in the disc $D(q/2,q/2)$, whence $ c_k=0$ for all $k$. This settles the case $n\ge 0$ and we turn our attention to $n<0$. We will use the identity
 \begin{equation}\label{jspeq202}
 e_{-n}(x) =-x^{-2/q} e_{n-1}(1/x)
 \end{equation}
 which holds for all $n\in\mathbb{N}$, and follows directly from the definition of $e_{-n}$.  Suppose that there exist complex numbers $\{c_k\}$ such that 
 $$
\left  \|\sum_{k=1}^\infty c_{k}e_{-k} \right \|= 0.
 $$
 Using (\ref{jspeq202}), it follows that
 $$
 \int_1^\infty \left|\sum_{k=1}^\infty c_{k} x^{-2/q} e_{k-1}(1/x)\right|^q\,dx= 0.
 $$
 With the substitution $t=1/x$ we obtain
 $$
 \int_0^1 \left|\sum_{k=1}^\infty c_{k} e_{k-1}(t)\right|^q\,dt= 0,
 $$
 so the result follows from the previous case.

 Our next step is to establish the desired estimate on the norm of $e_n$. To that end, we notice that the spectrum of $1-(2/q)C^*$ is the unit circle. Thus, if $\gamma\in(0,1)$, the spectral radius of $\gamma(1-(2/q)C^*)$ is less than one. It follows that this operator is similar to a strict contraction, hence power bounded. That is, there exists $K>0$ such that for $n\ge 0$, $\|(\gamma-(2\gamma/q)C^*)^n\|\le K$. Therefore,
 $$
 \|e_n\| =  \left \| \left(1-\frac2{q}C_\infty^*\right)^n e_0 \right \| \le 
 \left(\frac{1}{\gamma}\right)^n K \|e_0\|= K \left(\frac{1}{\gamma}\right)^n.
 $$
 As we had already noticed, $\|e_{-n}\|=\|e_{n-1}\|$ so the analogous estimate for $e_n$ indexed by negative integers follows.
 
 Finally, we will prove that $(1-2/{q} \,C_\infty^*)e_n=e_{n+1}$ for all $n\in\mathbb{Z}$. For $n\ge 0$ this is just the definition of $e_n$, so we focus on the case $n<0$.
 We will show that, for $n\ge 0$, 
 \begin{equation}
 \label{jspeq102}
 \left(1-\frac2{q}C_\infty^*\right) R \left(1-\frac2{q}C_\infty^*\right)e_n = Re_n.
 \end{equation}
 Once this is established the result will easily follow. Indeed, if $n>1$ then  
 \begin{align*}
 \left(1-\frac2{q}C_\infty^*\right)e_{-n}  & =  \left(1-\frac2{q}C_\infty^*\right) Re_{n-1} \\
 & =R \left(1-\frac2{q}C_\infty^*\right)^{-1}e_{n-1}\\
 & =  Re_{n-2} = e_{-n+1}.
 \end{align*}
 When $n=1$ 
 \begin{align*}
 \left(1-\frac2{q}C_\infty^*\right)e_{-1}(x) &=  \left(1-\frac2{q}C_\infty^*\right) Re_{0}(x) \\
 & =  -\left(1-\frac2{q}C_\infty^*\right) x^{-2/q} e_0\left(\frac{1}{x}\right) \\
 &=  -x^{-2/q} e_0\left(\frac{1}{x}\right) + \left(\frac{2}{q}\right)\int_x^\infty \frac{t^{-2/q}e_0(1/t)}{t}\,dt.
 \end{align*}
 Since $e_0=\chi_{(0,1)}$, if $0<x<1$ then $e_0(1/x)=0$ and the domain of integration is reduced to $(1,+\infty)$. Thus, we obtain
 $$
 \left(\frac{2}{q}\right)\int_1^\infty \frac{t^{-2/q}}{t}\,dt = 1.
 $$
 If $x\ge 1$ then $e_0(1/x)=1$ so we obtain
 $$
 -x^{-2/q}+ \left(\frac{2}{q}\right)\int_x^\infty \frac{t^{-2/q}}{t}\,dt=0.
 $$
 We conclude that $(1-2/{q} C_\infty^*)e_{-1}=e_0$. 
 
 Thus it remains to establish the identity (\ref{jspeq102}). 
 Let $f$ be any function in $L^q(0,\infty)$ that vanishes outside the interval $(0,1)$. Then
 \begin{align}
 \begin{gathered}\label{jspeq101}
 \left(1-\frac2{q}C_\infty^*\right) R \left(1-\frac2{q}C_\infty^*\right) f =\\= 
 - x^{-2/q} f\left(\frac{1}{x}\right) +
 \frac{2}{q} x^{-2/q} \int_{1/x}^\infty \frac{f(t)}{t}\,dt \\+
 \frac{2}{q} \int_x^\infty \frac{t^{-2/q}f(1/t)}{t}\,dt - \frac{4}{q^2}
 \int_x^\infty \frac{t^{-2/q}}{t} \,dt \int_{1/t}^\infty \frac{f(s)}{s}\,ds.
 \end{gathered}
 \end{align}
 If $0<x<1$ the first two terms are equal to 0, and in the remaining two, the domains of integration are changed. We obtain
 $$
 \frac{2}{q} \int_1^\infty \frac{t^{-2/q}f(1/t)}{t}\,dt - \frac{4}{q^2}
 \int_1^\infty \frac{t^{-2/q}}{t} \,dt \int_{1/t}^1 \frac{f(s)}{s}\,ds.
 $$
 Now the substitution $u=1/t$ followed by the change in the order of integration in the second term yields
 \begin{align*}
 & \frac{2}{q} \int_0^1 \frac{u^{2/q}f(u)}{u}\,du - \frac{4}{q^2}
 \int_0^1 \frac{u^{2/q}}{u} \,du \int_{u}^1 \frac{f(s)}{s}\,ds \\
 & =  \frac{2}{q} \int_0^1 \frac{u^{2/q}f(u)}{u}\,du - \frac{4}{q^2}
 \int_0^1 \frac{f(s)}{s}\,ds \int_0^s \frac{u^{2/q}}{u}\,du \\
 &= \frac{2}{q} \int_0^1 \frac{u^{2/q}f(u)}{u}\,du - \frac{4}{q^2}
 \int_0^1 \frac{f(s)}{s}\,\frac{q}{2}\,s^{2/q}\,ds = 0.
 \end{align*}
 If $x\ge 1$, we will obtain that all the terms in (\ref{jspeq101}) except for the first cancel. Once again, we use the substitution $u=1/t$ in the last two terms and obtain
 \begin{align}
 & - x^{-2/q} f\left(\frac{1}{x}\right) +
 \frac{2}{q} x^{-2/q} \int_{1/x}^\infty \frac{f(t)}{t}\,dt \\
 & + \frac{2}{q} \int_0^{1/x} \frac{u^{2/q}f(u)}{u}\,du - \frac{4}{q^2}
 \int_0^{1/x} \frac{u^{2/q}}{u} \,du \int_{u}^\infty \frac{f(s)}{s}\,ds.
 \end{align}
 Further, after interchanging the order of integration in the iterated integral, it becomes
 \begin{align*}
 & \int_0^{1/x} \frac{f(s)}{s}\,ds
 \int_0^s 
 \frac{u^{2/q}}{u} \,du +
 \int_{1/x}^\infty \frac{f(s)}{s}\,ds
 \int_0^{1/x} \frac{u^{2/q}}{u} \,du 
 \\
 & =
 \int_0^{1/x} \frac{f(s)}{s}\,\frac{q}{2}\,s^{2/q}\,ds
 +
 \int_{1/x}^\infty \frac{f(s)}{s}\,\frac{q}{2}\,x^{-2/q}\,ds,
 \end{align*}
 so it is easy to see that we have the announced cancelation. Combining these two cases we conclude that
 $$
 \left(1-\frac2{q}C_\infty^*\right) R \left(1-\frac2{q}C_\infty^*\right)f = Rf.
 $$
 whenever $f$ vanishes outside $(0,1)$.
 In particular, if $f=e_n$ for $n\ge 0$, we obtain (\ref{jspeq102}).
 \end{proof}

 \begin{proposition}\label{prop7.3}
 Let $\{e_n\}_{n\in\mathbb{Z}}$ and $\mathcal{F}$ be as in Proposition~\ref{propjsp101} and let $0<\theta<1.$ Let 
 $W_\theta:\mathcal{F}\to L^q(0,2\pi)$ be a linear transformation defined by 
  $$
  W_\theta e_n = \frac{\theta^{|n|}}{(1-\theta)^{\max\{1/p,1/q\}}}e^{int}, \mbox{ for }n\in\mathbb{Z},
  $$ 
 and extended linearly. Then there is a constant $K=K(p,q)$ such that, for any $\theta\in (0,1)$ and any $f\in \mathcal{F}$, $\|W_\theta f\|\le K\|f\|$.
 Consequently, $W_\theta$ extends to a bounded linear operator $W_\theta:L^q(0,\infty)\to L^q(0,2\pi)$.
 \end{proposition}
 \begin{proof}
 We will show that there exists
 such a constant $K$ that does not depend on $\theta$ and such that, for any $f=\sum_{k=-\infty}^\infty c_k e_k\in L^q(0,\infty)$ and any $n\in\mathbb{N}$,
 \begin{equation}\label{jspeq403}
 \left \| \sum\limits_{k=-n}^n c_k {\theta^{|k|}}e^{ikt} \right \| \le \frac{K}{(1-\theta)^{\max\{1/p,1/q\}}} \left \|\sum\limits_{k=-n}^n c_k e_k \right \|.
 \end{equation}
 We start with the fact that $\sum_{k=-n}^n c_k {\theta^{|k|}} e^{ikt}$ is continuous, so its modulus attains its maximum at some ${t_0}\in [0,2\pi]$. Consequently,
 \begin{align*}
 \label{jspeq203}
\left  \| \sum\limits_{k=-n}^n c_k {\theta^{|k|}}e^{ikt} \right \|^q &=  \int_0^{2\pi}  \left|\sum\limits_{k=-n}^n c_k {\theta^{|k|}} e^{ikt}\right|^q\,dt \\ 
 & \le  {2\pi}   \left|\sum\limits_{k=-n}^n c_k {\theta^{|k|}}  e^{ikt_0}\right|^q \\&= 2\pi \left|\sum\limits_{k=0}^n c_k \left(1-\frac{2}{q}\overline{\alpha}\right)^k
  + \sum\limits_{k=-n}^{-1} c_k \left(1-\frac{2}{q}\overline{\beta}\right)^{-k}
  \right|^q,
 \end{align*}
 where $\alpha = q/2(1- \theta e^{-it_0})$ and $\beta = q/2(1- \theta e^{it_0})$.
 Let 
 $$
 g_1(x) =\left(1-\frac{2}{q}\beta\right)^{-1}{\beta}\chi_{(0,1)}(x)\,x^{(1-\alpha)/\alpha},
 $$ 
 $$
 g_2(x)=-{\alpha}\chi_{(1,\infty)}(x)\,x^{-2/p-(1-\beta)/\beta} ,
 $$
 and $g=g_1+g_2$. 
 Notice that $g$
 belongs to $L^p(0,\infty)$.
 Indeed, it suffices to establish that 
 $$
 \mbox{Re}\left(\frac{p(1-\alpha)}{\alpha}\right)>-1\mbox{ and } 
 \mbox{Re}\left(-2-\frac{p(1-\beta)}{\beta}\right)<-1.
 $$
 These inequalities can be reduced to $\mbox{Re}(1/\alpha)>1/q$ and $\mbox{Re}(1/\beta)>1/q$,
 which  in turn is equivalent to  $\alpha,\beta\in D(q/2,q/2)$. Since these are obvious,
 $g\in L^p(0,\infty)$. Moreover, 
 \begin{align*}
 \|g\|^p &= \int_0^\infty |g_1+g_2|^p \\
 & =  \int_0^1 \left|1-\frac{2}{q}\beta\right|^{-p}|\beta x^{(1-\alpha)/\alpha}|^p + 
 \int_1^\infty |\alpha x^{-2/p-(1-\beta)/\beta}|^p \\&= \frac{1}{\theta^p}|\beta|^p \,\frac{1}{1+\mbox{Re}\,p\frac{1-\alpha}{\alpha}} +|\alpha|^p \,\frac{1}{1+\mbox{Re}\,p\frac{1-\beta}{\beta}}.
 \end{align*}
 Further,
 \begin{align*}
 1+\mbox{Re} \frac{p(1-\alpha)}{\alpha} &= 1-p+p\, \mbox{Re} \frac{1}{\alpha} \\
 & =  1-p+\frac{p}{|\alpha|^2}\mbox{Re}(\overline{\alpha})\\
 & = 1-p+\frac{2p}{q|1- \theta e^{-it_0}|^2} \mbox{Re}(1-\theta e^{it_0}) \\
 &=  \frac{p}{q} \left(-1+ 2\,\frac{1-\theta\cos t_0}{1-2\theta\cos t_0+\theta^2}\right) \\
 & =  \frac{p}{q} \frac{1-\theta^2}{1-2\theta\cos t_0+\theta^2},
 \end{align*}
 and the same equality holds with $\beta$ in place of $\alpha$.  Using the relation $\alpha=\overline{\beta}$, we obtain that
  \begin{eqnarray}
  \label{jspeq205}
  \|g\| &=  \left(\frac{1}{\theta^p}\frac{|\beta|^p}{1+\mbox{Re} \frac{p(1-\alpha)}{\alpha}}\right. +\left.  \frac{|\alpha|^p}{1+\mbox{Re} \frac{p(1-\beta)}{\beta}}\right)^{1/p} \\
 & = \left(\frac{1}{\theta^p}+1\right)^{1/p} {|\alpha|} \left(\frac{q}{p} \frac{1-2\theta\cos t_0+\theta^2}{1-\theta^2}\right)^{1/p}\\
 &= \left(\frac{(\theta^p+1) q}{(\theta+1)p}\right)^{1/p}\frac{|\alpha|^{1+2/p}}{\theta (1-\theta)^{1/p}}.
 \end{eqnarray}
 Next,
 \begin{small}
 \begin{align*}
& \left  \|\sum\limits_{k=-n}^n c_k e_k \right \| \\
&  \ge   \left|\int_0^\infty \sum\limits_{k=-n}^n c_k e_k\,\overline{g}\right|\,\frac1{\|g\|} \\
 &= \left|\int_0^\infty \sum\limits_{k=-n}^n c_k e_k\,\overline{g_1+g_2}\right|\,\frac1{\|g\|}\\
 &= \left|\int_0^\infty \sum\limits_{k=0}^n c_k e_k\,\overline{g_1}+\int_0^\infty \sum\limits_{k=-n}^{-1} c_ke_k \,\overline{g_2}\right|\,\frac1{\|g\|}\\
 &= \left|\int_0^\infty \sum\limits_{k=0}^n c_k \left(1-\frac2{q}C_\infty^*\right)^ke_0\,\overline{g_1}+\int_0^\infty \sum\limits_{k=-n}^{-1} c_k Re_{-k-1}\,\overline{g_2}\right|\,\frac1{\|g\|}\\
 &= \left|\int_0^\infty \sum\limits_{k=0}^n c_ke_0 \left(1-\frac2{q}C_\infty\right)^k\,\overline{g_1}+\int_0^\infty \sum\limits_{k=-n}^{-1} c_k R\left(1-\frac2{q}C_\infty^*\right)^{-k-1}e_0\,\overline{g_2}\right|\,\frac1{\|g\|}\\
 &= \left|\int_0^1 \sum\limits_{k=0}^n c_ke_0 \left(1-\frac{2}{q}\overline{\alpha}\right)^k\,\overline{g_1}+\int_0^\infty \sum\limits_{k=-n}^{-1} c_ke_0 \left(1-\frac2{q}C_\infty\right)^{-k-1}R^*\overline{g_2}\right|\,\frac1{\|g\|}.
 \end{align*}
 \end{small}
 It is not hard to see that the operator $R^*$ is given by the formula $R^*f(x) = -x^{-2/p}f(1/x)$, so 
 $$
 R^*\overline{g_2}(x) =\overline{\alpha} x^{-2/p} \chi_{(1,\infty)}(1/x)\,x^{2/p+\overline{(1-\beta)/\beta}} =
 \overline{\alpha}\chi_{(0,1)}(x)x^{\overline{(1-\beta)/\beta}}=\overline{\alpha} f_{\overline{\beta}}(x).
 $$
 Therefore,
 the second integral can be written as
 \begin{align*}
&  \overline{\alpha }\int_0^1 \sum\limits_{k=-n}^{-1} c_ke_0(x) \left(1-\frac2{q}C\right)^{-k-1}f_{\overline{\beta}}(x)\,dx\\
&=\overline{\alpha} \int_0^1 \sum\limits_{k=-n}^{-1} c_ke_0(x) \left(1-\frac2{q}\,\overline{\beta}\right)^{-k-1}f_{\overline{\beta}}(x)\,dx \\
&= \overline{\alpha}\overline{\beta} \sum\limits_{k=-n}^{-1} c_k \left(1-\frac2{q}\overline{\beta}\right)^{-k-1}\\
&= \overline{\alpha}\overline{\beta}\left(1-\frac2{q}\overline{\beta}\right)^{-1} \sum\limits_{k=-n}^{-1} c_k \left(1-\frac2{q}\overline{\beta}\right)^{-k}.
 \end{align*}
 Since the first integral equals 
 $$
 \overline{\alpha}\overline{\beta}\left(1-\frac2{q}\overline{\beta}\right)^{-1} \sum\limits_{k=0}^n c_k \left(1-\frac2{q}\overline{\alpha}\right)^k,
 $$
 we obtain that 
 \begin{small}
 \begin{align*}
\left \| \sum_{k=-n}^n c_ke_k \right \|&\ge \left|\overline{\alpha}\overline{\beta}\left(1-\frac2{q}\overline{\beta}\right)^{-1}\right|\,
 \left|\sum\limits_{k=0}^n c_k \left(1-\frac{2}{q}\overline{\alpha}\right)^k
  + \sum\limits_{k=-n}^{-1} c_k \left(1-\frac{2}{q}\overline{\beta}\right)^{-k}
  \right|\,\frac{1}{\|g\|} \\&=
  \frac{|\alpha|^2}{\theta}\,
  \frac{1}{(2\pi)^{1/q}} \|\sum_{k=-n}^n c_k \theta^{|k|} e^{ikt}\|
  \left(\frac{(\theta+1)p}{(\theta^p+1) q}\right)^{1/p}\frac{\theta (1-\theta)^{1/p}}{|\alpha|^{1+2/p}}\\&\ge
  \left(\frac{p}{q}\right)^{1/p} \frac{1}{(2\pi)^{1/q}} \,|\alpha|^{1-2/p}(1-\theta)^{1/p} \left \|\sum_{k=-n}^n c_k \theta^{|k|} e^{ikt} \right \|
 \end{align*}
 \end{small}
 If $1<p\le 2$ then $1-2/p\le 0$, so
 $$
 |\alpha|^{1-2/p} \ge \left(\frac{q}{2}\right)^{1-2/p}(1+\theta)^{1-2/p}>\left(\frac{q}{2}\right)
 ^{1-2/p} 2^{1-2/p}.
 $$
 If $p>2$ then $1-2/p> 0$, so
 $$
 |\alpha|^{1-2/p} \ge \left(\frac{q}{2}\right)^{1-2/p}(1-\theta)^{1-2/p}
 $$
 and it follows that, in this case,
 $$
 |\alpha|^{1-2/p} (1-\theta)^{1/p} \ge \left(\frac{q}{2}\right)^{1-2/p} (1-\theta)^{1-2/p+1/p} = 
 \left(\frac{q}{2}\right)^{1-2/p}
 (1-\theta)^{1/q}.
 $$
 Therefore,  there exists $K=K(p,q)$ such that
 (\ref{jspeq403}) holds.
 We conclude that $W$ is a bounded linear transformation and that $\|W\|\le K$.
 \end{proof}

 \begin{theorem}
 Let $C_\infty$ be the Cesaro operator on $L^p(0,\infty)$ for $1<p\le\infty$, and let $\lambda\ne 1$ be  a complex number. If $X$ is a bounded linear operator on $L^p(0,\infty)$ such that $C_\infty X=\lambda XC_\infty$, then $X=0$.
 \end{theorem}
 \begin{proof}
 Let $q$ be the exponent conjugate to $p$, i.e., $1/p+1/q=1$. Since $C_\infty$ acts on $L^p(0,\infty)$, its conjugate operator $C_\infty^*$ is a bounded operator acting on $L^q(0,\infty)$. Let $\{e_n\}_{n\in\mathbb{Z}}$ be set of functions inf $L^q(0,\infty)$ as defined above, let $\theta\in (0,1)$, and let $W=W_\theta$ be as in Proposition~\ref{prop7.3}.

 Next, let $M_z$ be the operator of multiplication by $e^{it}$ on $L^q(0,2\pi)$, and let $\Gamma$ be a weighted shift on $L^q(0,2\pi)$ with weight sequence $\{\mu_n\}$, i.e., 
 $$
 \Gamma e^{int}=\mu_n e^{i(n+1)t}, \mbox{ with } \mu_n=\begin{cases} \theta, &\text{if $n\ge 0$},\\ 1/\theta & \text{if $n<0$,}\end{cases} = \frac{\theta^{|n+1|}}{\theta^{|n|}}.
 $$
 Then
 \begin{align*}
 \left(1-\frac2{q}C_\infty^*\right)e_n & =W e_{n+1} = \frac{\theta^{|n+1|}}{(1-\theta)^{\max\{1/p,1/q\}}}e^{i(n+1)t} \\
 & =\frac{\mu_n \theta^{|n|}}{(1-\theta)^{\max\{1/p,1/q\}}} M_z e^{int} \\
 & = \Gamma We_n
 \end{align*}
 so $W(1-2/q\,C_\infty^*)=\Gamma W$. Further if $C_\infty X=\lambda XC_\infty$ then $X^*C_\infty^*=\overline{\lambda} C_\infty^*X^*$, so we have
 $$
 X^*\left(1-\frac{2}{q}C_\infty^*\right) =\left(1-\frac{2}{q}\overline{\lambda}C_\infty^*\right)X^*.
 $$
 This implies that $(1-2/q\overline{\lambda}C_\infty^*)X^* e_n = X^* (1-2/qC_\infty^*) e_n = X^* e_{n+1}$ and, inductively, that 
 \begin{equation}
 \label{jspeq501}
 X^*e_n = (1-\frac{2}{q}\overline{\lambda}C_\infty^*)^nX^* e_0,
 \end{equation} 
 for all $n\in\mathbb{Z}$. Notice that
 \begin{align*}
 W \left(1-\frac2{q}\overline{\lambda}C_\infty^*\right) & = W \left(1-\overline{\lambda}+\overline{\lambda}- \frac{2}{q}\overline{\lambda}C_\infty^*\right)\\
 & = (1-\overline{\lambda})W + \overline{\lambda} W \left(1-\frac{2}{q}C_\infty^*\right) = UW,
 \end{align*}
 where $U=1-\overline{\lambda}+\overline{\lambda}\Gamma$. By the definition of $\Gamma$, we have
 \begin{align*}
 Ue^{int} & = \left[(1-\overline{\lambda})+\overline{\lambda}\theta e^{it}\right]e^{int}, \mbox{ if }n\ge 0, \mbox{ and }\\
 Ue^{int} & = \left[(1-\overline{\lambda})+\overline{\lambda}\frac1{\theta} e^{it}\right]e^{int}, \mbox{ if }n< 0.
 \end{align*}
 The estimates established in Proposition~\ref{propjsp101} allow us to obtain an estimate on the operator norm $\|X^*\|$. We have
 \begin{align*}
 \|X^*\| & \ge \frac{\|X^*e_n\|}{\|e_n\|} \ge \frac{1}{K}\gamma^n \|X^* e_n\|, \mbox{ if $n\ge 0$, and }\\
 \|X^*\| & \ge \frac{\|X^*e_n\|}{\|e_n\|} \ge \frac{1}{K}\gamma^{-n} \|X^* e_n\|, \mbox{ if $n< 0$.}
 \end{align*}
 As for $\|X^*e_n\|$ we have 
 \begin{align*}
 \|X^*e_n\| &=\| \left(1-\frac{2}{q}\overline{\lambda}C_\infty^*\right)^n X^* e_0\| \\
 &\ge \frac{1}{\|W\|}\|W\left(1-\frac{2}{q}\overline{\lambda}C_\infty^*\right)^n X^* e_0\| \\
 &= \frac{1}{\|W\|} \|U^n WX^* e_0\|\\& = \frac{1}{\|W\|} \|U^n f\|
 \end{align*}
 where $f=WX^* e_0$. Combining with the previous estimates, we obtain that
 $$
 \|X^*\| \ge \frac{1}{K}\gamma^n \frac{1}{\|W\|} \|U^n f\|, \mbox{ if }
 n\ge 0,
 $$
 and
 $$
 \|X^*\| \ge \frac{1}{K\gamma^n} \frac{1}{\|W\|} \|U^n f\|, \mbox{ if }n<0.
 $$ 
 Let 
 \begin{align*}
 A_\gamma & =\{t\in [0,2\pi]: |\gamma(1-\overline{\lambda})+\overline{\lambda}\gamma\theta e^{it}|>1\},\\
 B_\gamma & =\{t\in [0,2\pi]: |\gamma^{-1}(1-\overline{\lambda})+\overline{\lambda}\gamma^{-1}/\theta e^{it}|<1\}.
 \end{align*} 
 Using the same argument as in the proof of Theorem~\ref{thm4.3}, we see that $f$ must be 0 on $A_\gamma\cup B_\gamma$. Since this must be true for any $\gamma\in (0,1)$, we see that $f$ must vanish on $A=\cup_{\gamma\in (0,1)}A_\gamma$ and $B=\cup_{\gamma\in (0,1)}B_\gamma$. Thus, $f$ can be different from 0 only on the complement of $A\cup B$. But,
 $$
 (A\cup B)^c = \{t\in [0,2\pi]: |(1-\overline{\lambda})+\overline{\lambda}\theta e^{it}|\le 1 \mbox{ and } |(1-\overline{\lambda})+\frac{\overline{\lambda}}{\theta} e^{it}|\ge 1\}.
 $$
 Let $re^{i\varphi}$ be the polar form of $(1-\overline{\lambda})/\overline{\lambda}$. Since we are assuming that $\lambda\ne 1$, this complex number is not zero, so $\varphi$ is well defined.  Then
 \begin{align*}
&  (A\cup B)^c \\
& = \{t\in [0,2\pi]:
 |r+\theta e^{i(t-\varphi)}|\le \frac1{|\lambda|}\mbox{ and } 
 |r+\frac{1}{\theta} e^{i(t-\varphi)}|\ge \frac1{|\lambda|}\}\\
 &= \{t\in [0,2\pi]: r^2+\theta^2 +2r\theta \cos(t-\varphi) \le \frac1{|\lambda|^2}\}  \cap \\
& \cap \{ t \in [0,2\pi] \colon  r^2+\frac{1}{\theta^2}+2\frac{r}{\theta}\cos(t-\varphi)\ge \frac1{|\lambda|^2}\} \\
 &= \{t\in [0,2\pi]: \frac{\theta}{2r} \left(\frac{1}{|\lambda|^2}-r^2-\frac{1}{\theta^2}\right)\le \cos (t-\varphi)\le 
 \frac{1}{2r\theta} \left(\frac{1}{|\lambda|^2}-r^2-{\theta^2}\right)\}.
 \end{align*}
Notice that, as $\theta\uparrow 1$, both bounds for $\cos (t-\varphi)$ converge to the same number. It follows that, for a fixed $t\in [0,2\pi]$ there exists $\Theta\in (0,1)$ such that, if $\theta\ge \Theta$ then $t\notin (A\cup B)^c$. In other words, if $\theta\ge \Theta$ then $f(t)=0$. 
 
 Let us write $X^*e_0 = \sum_{n\in\mathbb{Z}} c_ne_n$. Then 
 $$
 f(t)= (WX^* e_0) (t)=\sum_{n=-\infty}^\infty c_k \frac{\theta^{|n|}}{(1-\theta)^{\max\{1/p,1/q\}}}e^{int}.
 $$
 For a fixed $t\in [0,2\pi]$ the power series above is an analytic function of $\theta$, for $|\theta|<1$, and this function vanishes on the line segment $(\Theta,1)$, so it must be zero. Consequently, $c_{-n} e^{-int} + c_n e^{int} = 0$  for every $n\in\mathbb{N}$. Since this is true for all $t\in (A\cup B)^c$, it is easy to see that $c_n=0$ for all $n\in\mathbb{Z}$. Thus $X^*e_0=0$ and (\ref{jspeq501}) implies that $X^*e_n=0$ for all $n\in\mathbb{Z}$, whence $X=0$.
 \end{proof}

 \section{Some open problems}
 \label{problems}
 Here is a list of problems that we find interesting and that  we have not been able to solve.
\begin{enumerate}
\item Show that the co-analytic Toeplitz matrix \(A\) of Theorem \ref{main} induces a bounded linear operator on \(\ell^2,\) or in other words, show that the supremum in equation (\ref{sup})  is finite.
\item Show that if \(X\) is an extended eigenoperator for \(C_1\) on \(L^p[0,1]\) then there exists \(R \in \{C_1\}^\prime\) such that \(X=X_0R,\) where \(X_0\) is the weighted composition operator of Lemma \ref{thm:suffic}.
\item Show that if \(1<p<\infty\) and if \(\lambda\) is real and \(\lambda \geq 1\) then \(\lambda\) is an extended eigenvalue for \(C_0\) on \(\ell^p.\)
\item Let \(T \in {\mathcal B}(E)\) and consider the {\em Deddens algebra} \(\mathcal D_T\)  associated with  \(T,\)  that is, the family of all \(X \in {\mathcal B}(E)\) for which there is a constant \(M>0\) such that for every \(n \in \mathbb N\) and for every \(f \in E,\) 
\begin{eqnarray}
\|T^n X f \| \leq M \|T^nf\|.
\end{eqnarray}
When \(T\) is invertible this is equivalent to saying that 
\begin{eqnarray}
\label{eqn:deddens}
\sup_{n \in \mathbb N} \|T^n X T^{-n} \| <\infty.
\end{eqnarray}
The Deddens algebra \(\mathcal D_T\)  is a not necesarily closed subalgebra of \(\mathcal B(E)\) that contains all extended eigenoperators corresponding to extended eigenvalues \(\lambda\) with \(|\lambda| \leq 1.\) Show that \(\mathcal D_{C_\infty} = \{C_\infty\}^\prime.\) A consequence of this result would be that the set of extended eigenvalues for \(C_\infty\) reduces to \( \{1\}.\)
\end{enumerate}

\begin{acknowledgement}
The first author was partially supported  by  Ministerio de Econom\'{\i}a y Competitividad, Reino de Espa\~na,  under grant   MTM  2012-30748. The second author was partially supported by Junta de Andaluc\'{\i}a under grant FQM-257 and  by Vicerrectorado de investigaci\'on UCA. We are grateful to Giorgio Metafune and Luis Rodr\'{\i}guez-Piazza for useful input.
\end{acknowledgement}

\section*{References}
\bibliography{cesaro.bib}
\bibliographystyle{plain}

 \end{document}